%% file: neurips_2019.tex
\documentclass{article}

\usepackage[final]{neurips_2019}
\usepackage[utf8]{inputenc} % allow utf-8 input
\usepackage[T1]{fontenc}    % use 8-bit T1 fonts
\usepackage{hyperref}       % hyperlinks
\usepackage{url}            % simple URL typesetting
\usepackage{booktabs}       % professional-quality tables
\usepackage{amsfonts}       % blackboard math symbols
\usepackage{nicefrac}       % compact symbols for 1/2, etc.
\usepackage{microtype}      % microtypography
\usepackage{times}
\usepackage{bbm}
\usepackage{amsmath,amsthm}
\usepackage{xcolor}
\usepackage{graphicx}
\usepackage{subcaption}

\usepackage{cleveref}

\let\hat\widehat
\let\tilde\widetilde

\newcommand{\E}{\mbox{$\mathbb{E}$}}
\newcommand{\F}{\mbox{$\mathcal{F}$}}
\newcommand{\D}{\mbox{$\mathcal{D}$}}
\newcommand{\Tau}{\mathcal{T}}

\DeclareMathOperator*{\argmax}{argmax}

\newtheorem{fact}{Fact}

\newtheorem{theorem}{Theorem}
\newtheorem{definition}{Definition}
\newtheorem{lemma}[theorem]{Lemma}
\newtheorem{example}[theorem]{Example}
\newtheorem{corollary}[theorem]{Corollary}
\newtheorem{proposition}[theorem]{Proposition}
\newtheorem{remark}{Remark}

\title{Are sample means in multi-armed bandits \\ positively or negatively biased? }

\author{%
Jaehyeok Shin$^1$, Aaditya Ramdas$^{1,2}$ and Alessandro Rinaldo$^{1}$ \\
Department of Statistics and Data Science$^1$ \\
Machine Learning Department$^2$\\
Carnegie Mellon University \\
\texttt{\{shinjaehyeok, aramdas, arinaldo\}@cmu.edu}
}

\begin{document}

\maketitle

\begin{abstract}
It is well known that in stochastic multi-armed bandits (MAB), the sample mean of an arm is typically not an unbiased estimator of its true mean. In this paper, we decouple three different sources of this selection bias: adaptive \emph{sampling} of arms, adaptive \emph{stopping} of the experiment, and adaptively \emph{choosing} which arm to study.  Through a new notion called ``optimism'' that captures certain natural monotonic behaviors of algorithms, we provide a clean and unified analysis of how optimistic rules affect the sign of the bias. The main takeaway message is that optimistic sampling induces a negative bias, but optimistic stopping and optimistic choosing both induce a positive bias. These results are derived in a general stochastic MAB setup that is entirely agnostic to the final aim of the experiment (regret minimization or best-arm identification or anything else). We provide examples of optimistic rules of each type, demonstrate that simulations confirm our theoretical predictions, and pose some natural but hard open problems.
\end{abstract}

\section{Introduction} 
Mean estimation is one of the most fundamental problems in statistics. In the classic nonadaptive setting, we observe a fixed number of samples drawn i.i.d. from a fixed distribution with an unknown mean $\mu$. In this case, we know that the sample mean is an unbiased estimator of $\mu$.
% (while also having several other favorable properies like consistency and minimax optimality).
% In particular, it is unbiased, consistent, and converges almost surely to $\mu$. 
% Under tail assumptions such as sub-Gaussian or sub-exponential conditions, the sample mean is tightly concentrated around $\mu$. Lastly, the sample mean has minimax optimal risk with respect to suitable loss functions such as the Kullback-Leibler (KL) loss for distributions in a natural exponential family.

However, in many cases the data are collected and analyzed in an adaptive  manner, a prototypical example being the stochastic multi-armed bandits (MAB) framework \citep{robbins1952some}. During the data collection stage, in each round an analyst can draw a sample from one among a finite set of available distributions (arms) based on the previously observed data (\emph{adaptive sampling}). The data collecting procedure can also be terminated based on a data-driven stopping rule rather than at a fixed time (\emph{adaptive stopping}). Further, the analyst can choose a specific target arm based on the collected data (\emph{adaptive choosing}), for example choosing to focus on the arm with the largest empirical mean at the stopping time. 
% Lastly, in hindsight, the analyst may wonder what the bias of the sample mean of the chosen arm was at some past time (\emph{adaptive rewinding}).
In this setting, the sample mean is no longer unbiased, due to the selection bias introduced by all three kinds of adaptivity. In this paper, we provide a comprehensive understanding of the sign of the bias, decoupling the effects of these three sources of adaptivity. 

In a general and unified MAB framework, we first define natural notions of monotonicity (a special case of which we call ``optimism'') of sampling, stopping and choosing rules. Under no assumptions on the distributions beyond assuming that their means exist, we show that optimistic sampling provably results in  a negative bias, but optimistic stopping and optimistic choosing both provably result in a positive bias. Thus, the net bias can be positive or negative in general. This message is in contrast to a recent thought-provoking work by \cite{nie2018adaptively} titled \emph{``Why adaptively collected data has a negative bias...''} that is unfortunately misleading for practitioners, since it only analyzed the bias of adaptive sampling for a fixed arm at a fixed time.

As a concrete example, consider an offline analysis of data that was collected by an MAB algorithm (with any aim). Suppose that a practitioner wants to estimate the mean reward of some of the better arms that were picked more frequently by the algorithm. \cite{nie2018adaptively} proved that the sample mean of each arm is negatively biased under fairly common adaptive sampling rules. Although this result is applicable only to a fixed arm at a fixed time, it could instill a possibly false sense of comfort with  sample mean estimates since the practitioner might possibly think that sample means are underestimating the effect size. However, we prove that if the algorithm was adaptively stopped and the arm index was adaptively picked, then the net bias can actually be positive. Indeed, we prove that this is the case for the lil'UCB algorithm (Corollary~\ref{cor::lil'UCB}), but it is likely true more generally as captured by our main theorem. Thus, the sample mean may actually overestimate the effect size. This is an important and general phenomenon for both theoreticians (to study further and quantify) and for practitioners (to pay heed to) because if a particular arm is later deployed in practice, it may yield a lower reward than was possibly expected from the offline analysis. 

\paragraph{Related work and our contributions.} 
Adaptive mean estimation, in each of the three senses described above, has received much attention in both recent and past literature. Below, we discuss how our work relates to past work, proceeding one notion at a time in approximate historical order.

We begin by noting that a single-armed bandit is simply a random walk, where adaptive stopping has been extensively studied.
% , since even simplest of asymptotic questions are often nontrivial\footnote{For example, if a random walk is stopped at an increasing sequence of stopping times, the corresponding sequence of stopped sample means does \emph{not} necessarily converge to the true sample mean, even in probability, without regularity conditions on the distribution and stopping rules (see Ch.1 of \citet{gut2009stopped}). In contrast, recall that the strong law of large numbers implies that without adaptive stopping, the sample mean converges almost surely to the true mean.}. 
The book by \cite{gut2009stopped} on stopped random walks is an excellent reference, summarizing almost 60 years of advances in sequential analysis. 
% Some relevant authors include \citet{anscombe1952large,richter1965limit,starr1966asymptotic,starr1972further,siegmund1978estimation}, since they discuss inferential questions for stopped random walks or stopped tests, often in parametric and asymptotic settings. 
Most of these extensive results on random walks have not been extended to the MAB setting, which naturally involves adaptive sampling and choosing. Of particular relevance is the paper by \cite{starr1968remarks} on the sign of the bias under adaptive stopping, whose work is subsumed by ours in two ways: we not only extend their insights to the MAB setting, but even for the one-armed setting, our results generalize theirs.
% Motivated by this, we provide new consistency results, bounds for bias and risk that hold in the fully adaptive setting.

Characterizing the sign of the bias of the sample mean under adaptive sampling has been a recent topic of interest due to a surge in practical applications.
While estimating MAB ad revenues, \citet{xu2013estimation} gave an informal argument of why the sample mean is \emph{negatively} biased for ``optimistic'' algorithms. Later, \citet{villar2015multi} encountered this negative bias in a simulation study motivated by using MAB for clinical trials. Most recently, \citet{bowden2017unbiased} derived an exact formula for the bias and \citet{nie2018adaptively} formally provided conditions under which the bias is negative. Our results on ``optimistic'' sampling inducing a negative bias generalize the corresponding results in these past works.

Most importantly, however, these past results hold only at a predetermined time and for a fixed arm. Here, we put forth a complementary viewpoint that ``optimistic'' stopping and choosing induces a \emph{positive} bias. Indeed, one of our central conceptual contributions is an appropriate and crisp definition of ``monotonicity'' and ``optimism'' (Definition~\ref{def::monotone_sampling_strategy}), that enables a clean and general analysis. 

Our main theoretical result, Theorem~\ref{thm::bias_sign}, allows the determination of the sign of the bias in several interesting settings. Importantly, the bias may be of any sign when optimistic sampling, stopping and choosing are all employed together. We demonstrate the practical validity of our theory using some simulations that yield interesting insights in their own right.

% {\color{red}{to be completed}}

% \begin{enumerate}
% \item We prove that the bias is negative at fixed times for ``optimistic'' sampling or choosing rules, but it is positive for fixed arms at ``optimistic'' stopping times under nonadaptive sampling. Thus, under all kinds of adaptivity the bias can take either sign, depending on whether their combined effect is ``monotone''  and ).
% \end{enumerate}

The rest of this paper is organized as follows. In Section \ref{sec::acda_setting}, we briefly formalize the three notions of adaptivity by introducing a stochastic MAB framework. Section~\ref{sec::sign-bias} derives results on when the bias can be positive or negative. In Section~\ref{sec::smulations}, we demonstrate the correctness of our theoretical predictions through simulations in a variety of practical situations. We end with a brief summary in Section~\ref{sec::summary}, and for reasons of space, we defer all proofs to the Appendix. 

\section{The stochastic MAB framework} \label{sec::acda_setting}

Let  $P_1, \dots, P_K$ be $K$ distributions of interest (also called arms)  with finite means $\mu_k =  \E_{Y\sim P_k}[Y]$. 
% For the simplicity of presentation, we assume each arm has a univariate distribution but every statement in this paper can be straightforwardly extended to the multivariate setting. 
Every inequality and equality between two random variables is understood in the almost sure sense.

\subsection{Formalizing the three notions of adaptivity}
For those not familiar with MAB algorithms, \citet{lattimore2018bandit} is a good reference. 
The following general problem setup is critical in the rest of the paper:
%\begin{itemize}
%	\item For each time $t$, choose $A_t \sim \mathrm{Multi}(1, \nu_t(k \mid A_1, Y_1, \dots, A_{t-1}, Y_{t-1})) \in \{0,1, \dots, K\}$.
%	\item If $A_t = k \in [K]$, draw a sample $Y_t$ from the selected distribution $P_{k}$ 
%	\item If $A_t = 0$, set termination time $T = t -1$ and return $D = \{A_1, Y_1, \dots, A_T, Y_T \}$. 
%\end{itemize}
\begin{itemize}
    \item Let $W_{-1}$ denote all external sources of randomness that are independent of everything else. Draw an initial random seed $W_0 \sim U[0,1]$, and set $t=1$.

    \item At time $t$, let $\D_{t-1}$ be the data we have so far, which is given by
	\[
	 \D_{t-1}  :=\{A_1, Y_1, \dots, A_{t-1}, Y_{t-1}\},
	\]
	 where $A_s$ is the (random) index of arm sampled at time $s$ and $Y_s$ is the observation from the arm $A_s$. Based on the previous data (and possibly an external source of randomness), let  
	$\nu_t(k \mid \D_{t-1})  \in [0,1]$ be the conditional probability of sampling the $k$-th arm for all $k \in [K] := \{1,\ldots,K\}$ with $\sum_{k=1}^{K}\nu_t(k \mid \D_{t-1}) =1$.
	 Different choices for $\nu_t$ capture commonly used methods such as random allocation, $\epsilon$-greedy \citep{sutton1998introduction}, upper confidence bound algorithms \citep{auer2002finite, audibert2009minimax, garivier2011kl,kalyanakrishnan2012pac, jamieson_lil_2014} and Thompson sampling \citep{thompson1933likelihood, agrawal2012analysis,kaufmann2012thompson}. 
% 	\item At time $t$, for all $k \in [K]$, prescribe the probability 
% 	$\nu_t(k \mid \D_{t-1})  \in [0,1]$ of selecting the $k$-th arm 
% 	 based on the observed data (and possibly external randomness), where 
% 	 \[
% 	 \D_{t-1}  :=\{A_1, Y_1, \dots, A_{t-1}, Y_{t-1}\}. 
% 	 \]
% 	 Different choices for $\nu_t$ capture commonly used sampling methods including completely random allocation, $\epsilon$-greedy, upper confidence bound algorithms and Thompson sampling. 
	\item 
	If
% 		\textcolor{red}{$W_{t-1} \in \left(\sum_{j = 1}^{k-1} \nu_t(j), \sum_{j = 1}^{k} \nu_t(j)\right) \text{ for some $k \in [K]$,}$}
	$W_{t-1} \in \left(\sum_{j = 1}^{k-1} \nu_t(j\mid \D_{t-1}), \sum_{j = 1}^{k} \nu_t(j\mid \D_{t-1})\right) \text{ for some $k \in [K]$,}$
   then set $A_t = k$ which is equivalent to sample $A_t$ from a multinomial distribution with probabilities $\{\nu_t(k\mid\D_{t-1})\}_{k=1}^{K}$. Let $Y_t$ be a fresh independent draw from distribution $P_{k}$. This yields a natural filtration $\left\{ \mathcal{F}_t \right\}$ which is defined, starting with $\mathcal{F}_0 = \sigma\left(W_{-1},W_0\right)$, as
\[
\mathcal{F}_t := \sigma\left(W_{-1},W_0, Y_1, W_1, \dots,Y_t, W_t \right),~~\forall t \geq 1.
\]
Then, $\{Y_t\}$ is adapted to $\left\{ \mathcal{F}_t \right\}$, and $\{A_t\}, \{\nu_t\}$ are predictable with respect to $\left\{ \mathcal{F}_t \right\}$.
\item For each $k\in[K]$ and $t\geq 1$, define the running sum  and number of draws for arm $k$ as
$S_k(t) := \sum_{s=1}^{t} \mathbbm{1}(A_s = k) Y_s, ~~ N_k(t) := \sum_{s=1}^{t} \mathbbm{1}(A_s = k).$ Assuming that arm $k$ is sampled at least once, we define the  sample mean for arm $k$ as
\[
\hat{\mu}_k(t) := \frac{S_k(t)}{N_k(t)}.
\]
Then, $\{S_t\}$, $\{\hat{\mu}_k(t)\}$ are adapted to $\{\F_t\}$ and $\{N_k(t)\}$ is predictable with respect to $\{\F_{t}\}$.
	\item Let $\Tau$ be a stopping time with respect to $\{\F_t\}$. If $\Tau$ is nonadaptively chosen, it is denoted $T$.
	If $t < \Tau$, draw a random seed $W_{t} \sim U[0,1]$ for the next round, and increment $t$. Else return the collected data $\mathcal{D}_\Tau = \{A_1, Y_1, \dots, A_\Tau, Y_\Tau\} \in \F_\Tau$.
	\item After stopping, choose a data-dependent arm based on a possibly randomized rule $\kappa : \mathcal{D}_\Tau \cup \{W_{-1}\} \mapsto [K]$, but we denote the index $\kappa(\mathcal{D}_\Tau\cup \{W_{-1}\})$ as just $\kappa$ for short, so that the target of estimation is $\mu_\kappa$.  Note that $\kappa \in \F_\Tau$, but when $\kappa$ is nonadaptively chosen (is independent of $\F_\Tau$), we called it a fixed arm and denote it as $k$.
% 	\item We may adaptively rewind the clock to focus a previous time $\tau \leq \Tau$, if we wish to characterize the past behavior of a chosen sample mean $\hat{\mu}_\kappa(\tau)$. Note that $\tau$ is not a stopping time in general. If we do not adaptively rewind, then it corresponds to choosing $\tau=\Tau$.
\end{itemize}

% Depending on the choice of the selecting probability function $\nu_t$, the adaptive sampling strategy described above can be reduced to one of the 

% Throughout the paper, we say \emph{a data is collected under the adaptive sampling} if the data is collected from sub-$\psi$ distributions with an adaptive sampling strategy and a fixed time horizon $T$. If  a random stopping time $\Tau$ is used, we say \emph{a data is collected under the adaptive sampling and stopping}.  

% For each $k \in [K]$ and $t \in [\Tau]$, let $S_k(t)$ be the sum of random samples and $N_k(t)$ be the number of draws from distribution $P_k$ up to time $t$:

% Note that for each $k \in [K]$, the ``sum'' process $\{S_k(t)\}$ and the ``number of draws'' process $\{N_k(t)\}$ are respectively adapted and predictable with respect to the filtration $\{\mathcal{F}_t\}_{t \geq0}$. 
% %$S_k(t) \in \mathcal{F}_t$ and $N_k(t) \in \mathcal{F}_{t-1}$,  $\forall t \geq 1$.
% The sample mean estimator for the chosen arm is then defined as
% \begin{equation}
% \hat{\mu}_\kappa(\Tau) := \frac{S_\kappa(\Tau)}{N_\kappa(\Tau)}.
% \end{equation}
% Note that the sample mean can be defined only if $N_\kappa(\Tau) > 0$. Therefore,  we assume that the number of draws at the stopping time $N_\kappa(\Tau)$ is bounded away from $0$ throughout this paper.

 The phrase ``fully adaptive setting'' refers to the scenario of running an adaptive sampling algorithm until an adaptive stopping time $\Tau$, and asking about the sample mean of an adaptively chosen arm $\kappa$. When we are not in the fully adaptive setting, we explicitly mention what aspects are adaptive. 

\subsection{The tabular perspective on stochastic MABs} \label{sec::counterfactual_setting}
It will be useful to imagine the above fully adaptive MAB experiment using a $\mathbb{N} \times K$ table, $X^*_\infty$, whose rows index time and columns index arms. Here, we put an asterisk to clarify that it is counterfactual and not necessarily observable. We imagine this entire table to be populated even before the MAB experiments starts, where for every $i\in\mathbb{N},k\in[K]$, the $(i,k)$-th entry of the table contains an independent draw from $P_k$ called $X^*_{i,k}$. At each step, our observation $Y_t$ corresponds to the element $X^*_{N_k(t),A_t}$. Finally, we denote $\D^*_\infty =X^*_{\infty} \cup \{W_{-1}, W_0,\dots,W_t,\dots\}$.

Given the above tabular MAB setup (which is statistically indistinguishable from the setup described in the previous subsection), one may then find deterministic functions $f_{t,k}$ and $f^*_k$ such that
% $N_k(t) = \sum_{s \leq t} f^*_{k,s}(D^*_{s-1})$.
\begin{align}\label{eq::counterfactual-numberdraws}
N_k(\Tau) = \sum_{t\geq 1} \underbrace{\mathbbm{1}
\left(A_t = k \right)\mathbbm{1}(\Tau \geq t)}_{\F_{t-1}\text{-measurable}} = \sum_{t\geq 1} f_{t,k}(\D_{t-1})
\equiv  f^*_k (\D^*_\infty).
\end{align}
Specifically, the function $f_{t,k}(\cdot)$ evaluates to one if and only if we do not stop at time $t-1$, and pull arm $k$ at time $t$. Indeed, given $\D^*_\infty$, the stopping time $\Tau$ is deterministic and so is the number of times $N_k(\Tau)$ that a fixed arm $k$ is pulled, and this is what $f^*_k$ captures. Along the same lines, the number of draws from a chosen arm $\kappa$ at stopping time $\Tau$ can be written in terms of the tabular data as
\begin{equation}\label{eq:counterfactual-Nkappatau}
	N_\kappa(\Tau) ~=~ \sum_{k =1}^K \mathbbm{1}\left(\kappa = k\right) N_k(\Tau) ~\equiv~ \sum_{k =1}^k g^*_k  (\D^*_\infty) f^*_k (\D^*_\infty)
	\end{equation}
	for some deterministic set of functions $\{g_{k}^*\}$. Indeed, $g^*_k$ evaluates to one if after stopping, we choose arm $k$, which is a fully deterministic choice given $\D^*_\infty$.

% {\color{red}[Define $N_{\kappa}$ also here? Interpret and explain.]}

\section{The sign of the bias under adaptive sampling, stopping and choosing}\label{sec::sign-bias}

% Also, on the event $\Tau = \infty$, $\hat{\mu}_k(\Tau)$ and $N_k(\Tau)$ are understood as their almost sure pointwise limit which are well-defined by the strong law of large numbers and the monotonicity of $t \mapsto N_k(t)$. 
%From now, For the notationally simplicity, we will drop the dependency on $T$ for all related random variables if it is clear in the context.
%  Here, we provide characterizations of the sign of the bias when adaptive sampling, stopping and choosing are combined. 
 
\subsection{Examples of positive  bias due to ``optimistic'' stopping or choosing} \label{subSec::optimistic_rules}

%Under the nonadaptive setting where data is collected independently from a identical distribution, the sample mean is an unbiased estimator of the population mean.  Due to adaptivity in the data collecting procedure, however, the sample mean is no longer unbiased under the adaptive sampling \citep{xu2013estimation,nie2018adaptively}.  

In MAB problems, collecting higher rewards is a common objective of adaptive sampling strategies, and hence they are often designed to sample more frequently from a distribution which has larger sample mean than the others.  \citet{nie2018adaptively} proved that the bias of the sample mean for any \emph{fixed} arm and at any \emph{fixed} time is negative when the sampling strategy satisfies two conditions called ``Exploit'' and ``Independence of Irrelevant Options'' (IIO).  However, the emphasis on \emph{fixed} is important: their conditions are not enough to determine the sign of the bias under adaptive stopping or choosing, even in the simple nonadaptive sampling setting. Before formally defining our crucial notions of ``optimism'' in the next subsection, it is instructive to look at some examples.
\begin{example}\label{eg::bias-positive}
Suppose we continuously alternate between drawing a sample from each of two Bernoulli distributions with mean parameters $\mu_1, \mu_2 \in (0,1)$. This sampling strategy is fully deterministic, and thus it satisfies the Exploit and IIO conditions in \citet{nie2018adaptively}. For any fixed time $t$, the bias  equals zero for both sample means. Define a stopping time $\Tau$ as the first time we observe $+1$ from the first arm.  Then the sample size of the first arm, $N_1(\Tau)$, follows a geometric distribution with parameter $\mu_1$, which implies that the bias of $\hat{\mu}_1(\Tau)$ is
\begin{align*}
\mathbb{E}\left[\hat{\mu}_1 (\Tau) - \mu_1 \right]= \mathbb{E}\left[\frac{1}{N_1(\Tau)}\right] - \mu_1 = \frac{\mu_1 \log(1/ \mu_1)}{1-\mu_1} - \mu_1,
\end{align*}
which is positive for all $\mu_1 \in (0,1)$. 
\end{example}
This example shows that for nonadaptive sampling, adaptive stopping can induce a \emph{positive} bias. In fact, this example is not atypical, but is an instance of a more general phenomenon explored in the one-armed setting in sequential analysis.  For example, \citet[Ch. 3]{siegmund1978estimation} contains the following classical result for a Brownian motion $W(t)$ with positive drift $\mu > 0$.
\begin{example}\label{eg::brownian-stopping}
If we define a stopping time as the first time $W(t)$ exceeds a line with slope $\eta$ and intercept $b>0$, that is $\Tau_B :=\inf\{t\geq 0: W(t) \geq \eta t + b \}$, then for any slope $\eta \leq \mu$, we have $\mathbb{E}\left[\frac{W(\Tau_B)}{\Tau_B} - \mu\right] = 1/b$. Note that a sum of Gaussians with mean $\mu$ behaves like a time-discretization of a Brownian motion with drift $\mu$; since $\mathbb{E}W(t) = t\mu$, we may interpret $W(\Tau_B)/\Tau_B$ as a stopped sample mean, and the last equation implies that its bias is $1/b$, which is positive.
\end{example}
 Generalizing further, \citet{starr1968remarks} proved the following remarkable result. 
\begin{example}\label{eg::starr-woodroofe-stopping}
If we stop when the sample mean crosses any predetermined upper boundary, the stopped sample mean is always positive biased (whenever the stopping time is a.s. finite). Explicitly, choosing any arbitrary sequence of real-valued constants $\{c_k\}$, define $\Tau_c := \inf\{t: \hat{\mu}_1(t) > c_t \}$, then as long as the observations $X_i$ have a finite mean and $\Tau_c$ is a.s. finite, we have $\mathbb{E}\left[\hat{\mu}_1 (\Tau_c) - \mu_1 \right] - \mu_1 > 0$. 
\end{example}
Surprisingly, we will generalize the above strong result even further. Additionally, stopping times in the MAB literature can be thought of as extensions of $\Tau_c$ and $\Tau_B$ to a setting with multiple arms, and we will prove that indeed the bias induced will still be positive.
% instead of stopping when an arm's empirical mean crosses a time-dependent threshold, MAB algorithms often stop when the difference between the empirical means of two arms (or more pairs) crosses some time-dependent threshold. 
We end with an example of the positive bias induced by ``optimistic'' choosing: 
\begin{example}\label{eg::argmax}
Given $K$ standard normals $\{Z_i\}$ (to be thought of as one sample from each of $K$ arms), let $\kappa = \argmax_{k} Z_k$, that is, we choose the arm with the largest observation. It is well known that $\mathbb{E}\left[Z_{\kappa}\right] = \mathbb{E}\left[\max_{k \in [K]} Z_k\right] \asymp \sqrt{2\log K}$. Since $\E{Z_k}=0$ for all $k$, but $\E{Z_\kappa}>0$, the ``optimistic'' choice $\kappa$ induces a positive bias.
\end{example}

% \begin{example}\label{eg::argmax}
% Given $K$ standard normals $\{Z_i\}$ (to be thought of as one sample from each of $K$ arms), we know that $\mathbb{E}\left[\max_{k \in [K]} Z_k\right] \asymp \sqrt{2\log K}$. Since $\E{Z_k}=0$ for all $k$, but $\E{Z_\kappa}>0$ for $\kappa = \argmax_{k} Z_k$, \textcolor{red}{explain kappa in english} optimistic choosing induces a positive bias.
% \end{example}

In many typical MAB settings, we should expect sample means to have two contradictory sources of bias: negative bias from ``optimistic  sampling'' and positive bias from ``optimistic  stopping/choosing''.

\subsection{Positive or negative bias under monotonic sampling, stopping and choosing}

% Since the adaptive sampling, stopping and choosing act as different directional sources of bias, we need to take all into account in order to characterize the bias of the sample mean.

% Recall from Section~\ref{sec::counterfactual_setting} that the number of draws from the chosen arm $\kappa$ at stopping time $\Tau$ can be written in terms of the counterfactual data as
% \begin{equation}
% 	N_\kappa(\Tau) =\sum_{k =1}^K \mathbbm{1}\left(\kappa = k\right) N_k(\Tau) := \sum_{k =1}^k g^*_k  (\D^*_\infty) f^*_k (\D^*_\infty).
% 	\end{equation}
% 	for some deterministic  functions $\{f_{k}^*\}$ and $\{g_{k}^*\}$ where $\D_{\infty}^* $ is the counterfactual data.  
	
Based on the expression \eqref{eq:counterfactual-Nkappatau}, we formally state a characteristic of data collecting strategies which fully determines the sign of the bias as follows.

\begin{definition} \label{def::monotone_sampling_strategy}
A data collecting strategy is ``monotonically increasing (or decreasing)'' if for any $i \in \mathbb{N}$ and $k \in [K]$, the function  $\D^*_\infty \mapsto  g^*_k  (\D^*_\infty) / f^*_k (\D^*_\infty) \equiv \mathbbm{1}\left(\kappa = k\right) / N_k(\Tau)$, is an increasing (or decreasing) function of $X_{i,k}^*$  while keeping all other entries in $\D^*_\infty$ fixed. Further, we say that
\begin{itemize}
    \item a data collecting strategy has an optimistic sampling rule if the function  $\D^*_\infty \mapsto  N_k(t)$ is an increasing function of $X_{i,k}^*$  while keeping all other entries in $\D^*_\infty$ fixed for any fixed $i \in \mathbb{N}$, $t \geq 1$ and $k \in [K]$;
    \item  a data collecting strategy has an optimistic stopping rule if $\D^*_\infty \mapsto  \Tau$ is a decreasing function of $X_{i,k}^*$ while keeping all other entries in $\D^*_\infty$ fixed for any fixed $i \in \mathbb{N}$ and $k \in [K]$;
    \item  a data collecting strategy has an optimistic choosing rule if $\D^*_\infty \mapsto \mathbbm{1}(\kappa =k)$ is an increasing function of $X_{i,k}^*$ while keeping all other entries in $\D^*_\infty$ fixed for any fixed $i \in \mathbb{N}$ and $k \in [K]$.
\end{itemize}
\end{definition}

Note that if a data collecting strategy has an optimistic sampling (or stopping or choosing) rule, with the other components being nonadaptive, then the strategy is monotonically decreasing (increasing).
We remark that nonadaptive just means independent of the entries $X^*_{i,k}$, but it is not necessarily deterministic\footnote{An example of a random but nonadaptive stopping rule: flip a (potentially biased) coin at each step to decide whether to stop. 
An example of a random but nonadaptive sampling rule: with probability half pick a uniformly random arm, and with probability half pick the arm that has been sampled most often thus far.
}. 
The above definition warrants some discussion to provide intuition.

% Roughly speaking, the optimistic stopping condition requires that if a sample from the $k$-th distribution was increased while keeping all other values fixed, then the data collecting strategy would draw a smaller number of samples from the $k$-th distribution up to the (altered) stopping time. 
Roughly speaking, under optimistic stopping, if a sample from the $k$-th distribution was increased while keeping all other values fixed, the algorithm would reach its termination criterion sooner.
For instance,  $\Tau_B$ from Example~\ref{eg::brownian-stopping} and the criterion in Example~\ref{eg::bias-positive} are both optimistic stopping rules. Most importantly, boundary-crossing is optimistic:
\begin{fact}\label{fact:starr}
The general boundary-crossing stopping rule of \citet{starr1968remarks}, denoted $\Tau_c$ in Example~\ref{eg::starr-woodroofe-stopping}, is an optimistic stopping rule (and hence optimistic stopping is a weaker condition).
\end{fact}

Optimistic stopping rules do not need to be based on the sample mean; for example, if $\{c_t\}$ is an arbitrary sequence, then $\Tau := \inf\{t \geq 3: X_t + X_{t-2} \geq c_t \}$ is an optimistic stopping rule. In fact, $\Tau_\ell := \inf\{t \geq 3: \ell_t(X_{1},\dots,X_t) \geq c_t \}$ is optimistic, as long as each $\ell_t$ is coordinatewise nondecreasing.

For optimistic choosing, the previously discussed argmax rule (Example~\ref{eg::argmax}) is  optimistic. More generally, it is easy to verify the following: 
\begin{fact}\label{fact:choice}
For any probabilities $p_1\geq p_2\dots \geq p_K$ that sum to one, a rule that chooses the arm with the $k$-th largest empirical mean with probability $p_k$, is an optimistic choosing rule. 
\end{fact}
Turning to the intuition for optimistic sampling, if a sample from the $k$-th distribution was increased while keeping all other values fixed, the algorithm would sample the $k$-th arm more often. We claim that optimistic sampling is a weaker condition than the Exploit and IIO conditions employed by \cite{nie2018adaptively}. 
% \\ JH: we may need to weaken this argument since the proof of claim is now the part of the proof of theorem 1 in \cite{nie2018adaptively}. Thus I used the "fact" environment instead of claim.}

\begin{fact}\label{fact:sampling}
\label{claim::exploit_imply_optimistic}
 The ``Exploit'' and ``IIO'' conditions in \citet{nie2018adaptively} together imply that the sampling strategy is optimistic (and hence optimistic sampling is a weaker condition). Further, as summarized in Appendix~\ref{Appen::sampling}, $\epsilon$-greedy, UCB and Thompson sampling (Gaussian-Gaussian and Beta-Bernoulli, for instance) are all optimistic sampling methods. 
\end{fact}
For completeness, we prove the first part formally in Appendix~\ref{Appen::subSec::proof_claim}, which builds heavily on observations already made in the proof of Theorem~1 in \citet{nie2018adaptively}. 
Beyond the instances mentioned above, Corollary~\ref{cor::sufficient_TS} in the supplement captures a sufficient condition for Thompson sampling with one-dimensional exponential families and conjugate priors to be optimistic.
We now provide an expression for the bias that holds at any stopping time and for any sampling algorithm.

% \textcolor{red}{Additionally, if the adaptive sampling and stopping rules satisfy the  monotonically increasing (or decreasing) condition, we show that that the covariance between $\hat{\mu}_k(\Tau)$ and $N_k(\Tau)$ is always non-negative (or non-positive) and that the sample mean is negatively (or positively) biased.}
% \textcolor{red}{
% \begin{theorem} \label{thm::bias_expression_and_sign}
% 	For each fixed $k \in [K]$, let $\Tau$ be a stopping time with respect to the natural filtration $\{\mathcal{F}_t\}$. If $0 < \mathbb{E}N_k(\Tau) < \infty$, the bias of $\hat{\mu}_k(\Tau)$ is given as
% 	\begin{align} \label{eq:bias}
% 	\E\left[ \hat{\mu}_k(\Tau) - \mu_k \right] =- \frac{\mathrm{Cov}\left(\hat{\mu}_k(\Tau), N_k(\Tau)\right)}{\E\left[ N_k(\Tau)\right]}. 
%     \end{align}
%     If the data collecting strategy is monotonically increasing with respect to  the $k$-th distribution, for example under optimistic sampling, then we have
% 	\begin{equation} \label{eq::sign_of_cov_bias_increaing}
% 	\mathrm{Cov}\left(\hat{\mu}_k(\Tau), N_k(\Tau)\right) \geq 0~~\text{and}~~\E\left[ \hat{\mu}_k(\Tau) - \mu_k \right] \leq 0. 
% 	\end{equation}
% 	Similarly if the data collecting is monotonically decreasing with respect to the $k$-th distribution, for example under optimistic stopping, then we have
% 	\begin{equation} \label{eq::sign_of_cov_bias_decreasing}
% 	\mathrm{Cov}\left(\hat{\mu}_k(\Tau), N_k(\Tau)\right) \leq 0~~\text{and}~~\E\left[ \hat{\mu}_k(\Tau) - \mu_k \right] \geq 0. 
% 	\end{equation}
% \end{theorem}
% }

\begin{proposition} \label{prop::bias_expression}
	Let $\Tau$ be a stopping time with respect to the natural filtration $\{\mathcal{F}_t\}$. For each fixed $k \in [K]$ such that $0 < \mathbb{E}N_k(\Tau) < \infty$, the bias of $\hat{\mu}_k(\Tau)$ is given as
	\begin{align} \label{eq:bias}
	\E\left[ \hat{\mu}_k(\Tau) - \mu_k \right] =- \frac{\mathrm{Cov}\left(\hat{\mu}_k(\Tau), N_k(\Tau)\right)}{\E\left[ N_k(\Tau)\right]}. 
	\end{align}
\end{proposition}
The proof may be found in Appendix~\ref{subSec::prop_bias_expression}. A similar expression  was derived in \cite{bowden2017unbiased}, but only for a fixed time $T$.
In order to extend it to stopping times (that are allowed to be infinite, as long as $\mathbb{E}N_k(\Tau) < \infty$), we derive a simple generalization of Wald's first identity to the MAB setting. Specifically, recalling that $S_k(t) = \hat{\mu}_k(t) N_k(t)$, we show the following:
\begin{lemma}\label{lem:Wald-gen}
Let $\Tau$ be a stopping time with respect to the natural filtration $\{\mathcal{F}_t\}$.  For each fixed $k \in [K]$ such that $\mathbb{E}N_k(\Tau) < \infty$, we have
$
\E[S_k(\Tau)] = \mu_k \E[N_k(\Tau)].
$
\end{lemma}
This lemma is also proved in Appendix \ref{subSec::prop_bias_expression}.
Proposition \ref{prop::bias_expression} provides a simple, and somewhat intuitive, expression of the bias for each arm. It implies that if the covariance of the sample mean of an arm and the number of times it was sampled is positive (negative), then the bias is negative (positive). We now formalize this intuition below, including for  adaptively chosen arms. The following theorem shows that if the adaptive sampling, stopping and choosing rules are monotonically increasing (or decreasing), then the sample mean is  positively (or negatively) biased.
\begin{theorem} \label{thm::bias_sign}
	Let $\Tau$ be a stopping time with respect to the natural filtration $\{\mathcal{F}_t\}$ and let $\kappa : \D_{\Tau} \mapsto [K]$ be a choosing rule. Suppose each arm has finite expectation and, for all $k$ with $\mathbb{P}\left(\kappa =k\right) > 0$, we have $\mathbb{E}\left[N_k(\Tau)\right] < \infty$ and $N_k(\Tau) \geq 1$. If the data collecting strategy is monotonically decreasing, for example under optimistic sampling with nonadaptive stopping and choosing, then we have 
	\begin{equation} \label{eq::sign_of_cov_bias_increasing_each_k}
	\E\left[ \hat{\mu}_\kappa (\Tau)\mid \kappa = k \right] \leq \mu_k,~~\forall k: \mathbb{P}(\kappa = k) > 0,
	\end{equation}
    which also implies that
    \begin{equation}\label{eq::sign_of_cov_bias_increasing}
     \E\left[ \hat{\mu}_\kappa(\Tau) - \mu_\kappa \right] \leq 0.  \end{equation} 
	Similarly if the data collecting strategy is monotonically increasing, for example under optimistic stopping with nonadaptive sampling and choosing, or under optimistic choosing with nonadaptive sampling and stopping, then we have 
	\begin{equation} \label{eq::sign_of_cov_bias_decreasing_each_k}
	\E\left[ \hat{\mu}_\kappa(\Tau)\mid \kappa = k \right] \geq \mu_k,~~\forall k: \mathbb{P}(\kappa = k) > 0,
	\end{equation}
	which also implies that
	 \begin{equation}\label{eq::sign_of_cov_bias_decreasing}
     \E\left[ \hat{\mu}_\kappa(\Tau) - \mu_\kappa \right] \geq 0.  \end{equation} 
	If each arm has a bounded distribution then the condition $\mathbb{E}\left[N_k(\Tau)\right] < \infty$ can be dropped. 
\end{theorem}
\begin{remark} \label{remark::moment_condition}
    In fact, if each arm has a finite $p$-th moment for a fixed $p > 2$ then the condition $\mathbb{E}\left[N_k(\Tau)\right] < \infty$ can be dropped. 
    % It is an open question whether we can show the argument holds under a weaker moment condition. 
\end{remark}

The proofs of \Cref{thm::bias_sign} and \Cref{remark::moment_condition} can be found in \Cref{sec:Proof_of_bias_sign} and are based on martingale arguments that are quite different from the ones used in  \cite{nie2018adaptively}. See also Appendix~\ref{sec::intuition} for an intuitive explanation of the sign of the bias under optimistic sampling, stopping or choosing rules.
The expression~\eqref{eq:bias} intuitively suggests situations when the sample mean estimator $\hat{\mu}_k(\Tau)$ is biased, while the inequalities in \eqref{eq::sign_of_cov_bias_increasing_each_k} and~\eqref{eq::sign_of_cov_bias_decreasing_each_k} determine the direction of bias under the monotonic or optimistic conditions. Due to Facts \ref{fact:starr}, \ref{fact:choice} and \ref{fact:sampling}, several existing results are immediately subsumed and generalized by Theorem~\ref{thm::bias_sign}. Further, the following corollary is a particularly interesting special case dealing with the lil'UCB algorithm by \cite{jamieson_lil_2014} which uses adaptive sampling, stopping and choosing, as summarized in Section~\ref{subSec::simulation_4}.

\begin{corollary} \label{cor::lil'UCB}
The lil'UCB algorithm is a monotonically increasing strategy, and thus the sample mean of the reported arm when  lil'UCB  stops is always positively biased.
\end{corollary}

The proof is described in Appendix~\ref{Appen::subSec::lilUCB_proof}. The above result is interesting because of the following reasons: (a) when viewed separately, the sampling, stopping and choosing rules of the lil'UCB algorithm all seem to be optimistic (however, they are not optimistic, because our definition requires two out of three to be nonadaptive); hence it is apriori unclear which rule dominates and whether the net bias should be positive or negative; (b) we did not have to alter anything about the algorithm in order to prove that it is a monotonically increasing strategy (for any distribution over arms, for any number of arms). The generality of the above result showcases the practical utility of our theorem, whose message is in sharp contrast to the title of the paper by \cite{nie2018adaptively}.

% We remark that several special cases of the above theorem are already more general than previously known results. For example, .
Next, we provide simulation results that verify that our monotonic and optimistic conditions accurately capture the sign of the bias of the sample mean.

\section{Numerical experiments} \label{sec::smulations}

\subsection{Negative bias from optimistic sampling rules in multi-armed bandits} \label{subSec::simul_1}
% One of common MAB objectives is to minimize regret which is defined as the expected difference between the sum of rewards (observations) from the optimal arm and the sum of the collected rewards by a chosen algorithm. To achieve this goal, 
Recall Fact~\ref{fact:sampling}, which stated that common MAB adaptive sampling strategies like greedy (or $\epsilon$-greedy), upper confidence bound (UCB) and Thompson sampling are optimistic. 
Thus, for a deterministic stopping time, Theorem~\ref{thm::bias_sign} implies that the sample mean of each arm is always negatively biased. 
To demonstrate this, we conduct a simulation study in which we have three unit-variance Gaussian arms with $\mu_1 = 1, \mu_2 = 2$ and $\mu_3 =3$. After sampling once from each arm, greedy, UCB and Thompson sampling are used to continue sampling until $T = 200$. We repeat the whole process from scratch  $10^4$ times for each algorithm to get an accurate estimate for the bias.\footnote{In all experiments, sizes of reported biases are larger than at least 3 times the Monte Carlo standard error.}
Due to limited space, we present results from UCB and Thompson sampling only but detailed configurations of algorithms and a similar result for the greedy algorithm can be found in Appendix~\ref{appen::subSec::simul_1}. 
Figure~\ref{fig::simul_1_UCB_TS} shows the distribution of observed differences between sample means and the true mean for each arm. Vertical lines correspond to biases. The example demonstrates that the sample mean is negatively biased under optimistic sampling rules. 

\begin{remark}
The main goal in our simulations is to visualize and corroborate our theoretical results about the sign of the bias. As a result, we do not make any attempt to optimize the parameters for UCB or Thompon sampling for the purpose of minimizing the regret, since the latter is not the paper's aim. However, investigating the relationship between the performance of MAB algorithms and the bias at the time horizon would be an interesting future direction of research.
\end{remark}

\begin{figure}[h!]
	\begin{center}
	\includegraphics[scale =  0.6]{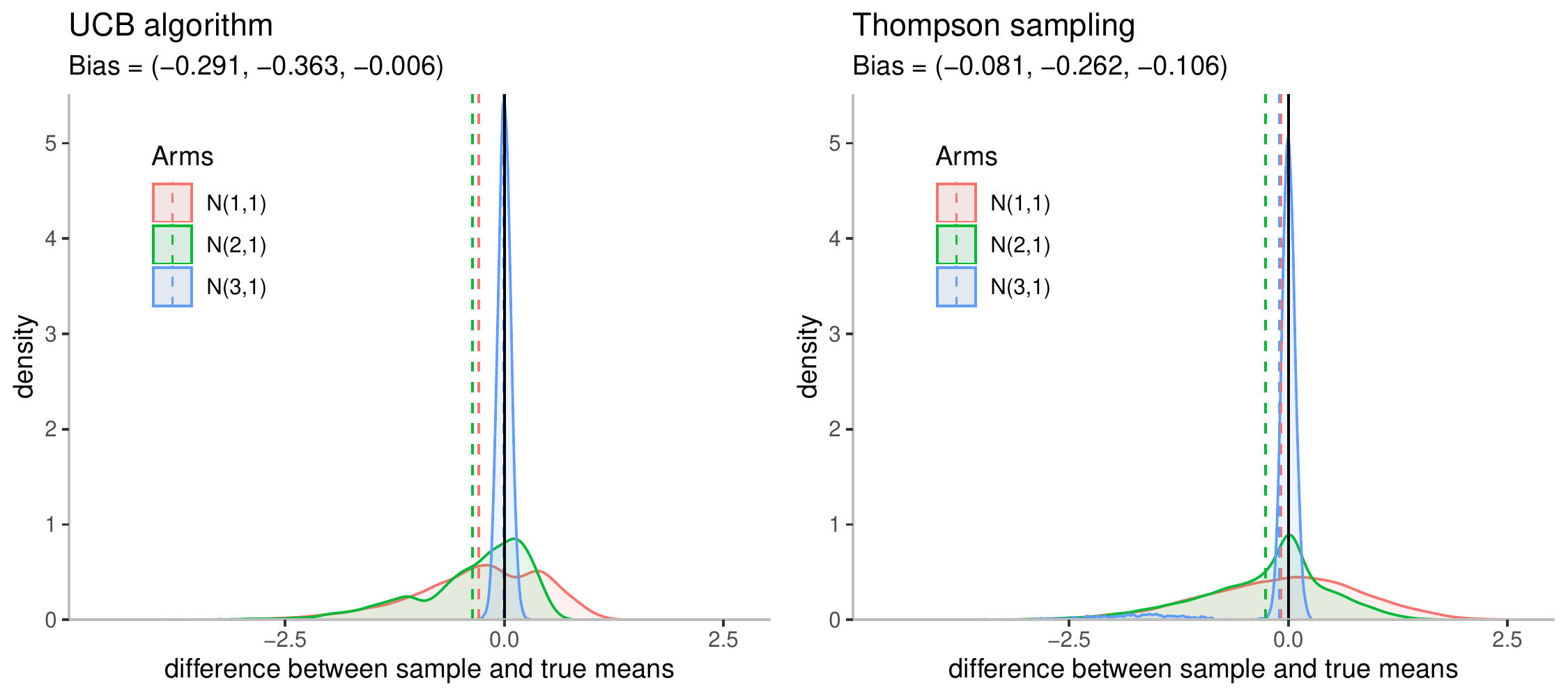}
	\end{center}
	\caption{\em  Data is collected by UCB (left) and Thompson sampling (right) algorithms from three unit-variance Gaussian arms with $\mu_1 =1, \mu_2 = 2$ and $\mu_3 = 3$. For all three arms, sample means are negatively biased (at fixed times). A similar result for the greedy algorithm can be found in Appendix~\ref{appen::subSec::simul_1}.}
	\label{fig::simul_1_UCB_TS}
\end{figure}

%\textcolor{red}{TODO::Add some remark or discussion about simulations. Especially about the severeness of bias and skewness of distributions.}

\subsection{Bias from stopping a one-sided sequential likelihood ratio test} \label{subSec::simulation_2}

Suppose we have two independent sub-Gaussian arms with common and known parameter $\sigma^2$ but unknown means $\mu_1$ and $\mu_2$. Consider the following testing problem:
\[
H_0 : \mu_1 \leq \mu_2~~\textrm{vs}~~ H_1 : \mu_1 > \mu_2.
\]
To test this hypothesis, suppose we draw a sample from arm $1$ for every odd time and from arm $2$ for every even time. Instead of conducting a test at a fixed time, we can use the following one-sided sequential likelihood ratio test \citep{robbins1970statistical, howard2018uniform}: for any fixed $w > 0$ and $\alpha \in (0,1)$, define a stopping time $\Tau$ as
\begin{equation}
    \Tau^w:=\inf\left\{t \in \mathbb{N}_{\mathrm{even}} : \hat{\mu}_1(t) - \hat{\mu}_2(t) \geq \frac{2\sigma}{t} \sqrt{\left(t + 2w\right)\log \left(\frac{1}{2\alpha}\sqrt{\frac{t + 2w}{2w}}+1\right)}  \right\},
\end{equation}
where $\mathbb{N}_{\mathrm{even}} := \{2n : n \in \mathbb{N}\}$. For a given fixed maximum even time $M \geq 2$, we stop sampling at time $\Tau_M^w := \min\left\{\Tau^w, M\right\}$. Then, we reject the null $H_0$ if $\Tau_M^w < M$. It can be checked \citep[Section~8]{howard2018uniform} that, for any fixed $w >0$, this test controls the type-1 error at level $\alpha$ and the power goes to $1$ as $M$ goes to infinity.  

%Whether we reject the null or nor, we may wonder the size of the mean for each arm and the sample mean is one of the most commonly used estimator.
For the arms $1$ and $2$, these are special cases of optimistic and pessimistic stopping rules respectively. From Theorem~\ref{thm::bias_sign}, we have that $\mu_1 \leq \mathbb{E}\hat{\mu}_1(\Tau_M^w)$  and $\mu_2 \geq \mathbb{E}\hat{\mu}_2(\Tau_M^w)$.
% In fact, if we treat each difference between samples from two arms at time $t$ and $t+1$ for all $t\geq1$ then we can further show that
% \[
% \mathbb{E}\hat{\mu}_2(\Tau_M) - \mu_2 \leq 0 \leq \mathbb{E}\hat{\mu}_1(\Tau_M) - \mu_1,
% \]
% which can be also viewed a direct application of \citet{starr1968remarks}.
To demonstrate this, we conduct two simulation studies with unit variance Gaussian errors: one under the null hypothesis  $(\mu_1, \mu_2) = (0,0)$, and one under the alternative hypothesis $(\mu_1, \mu_2) = (1,0)$. We choose $M = 200$, $w = 10$ and $\alpha = 0.1$. As before, we repeat each experiment $10^4$ times for each setting. Figure~\ref{fig::simul_2} shows the distribution of observed differences between sample means and the true mean for each arm under null and alternative hypothesis cases. Vertical lines correspond to biases. The simulation study demonstrates that the sample mean for arm $1$ is positively biased and the sample mean for arm $2$ is negatively biased as predicted.

\begin{figure}[h!]
	\begin{center}
	\includegraphics[scale =  0.3]{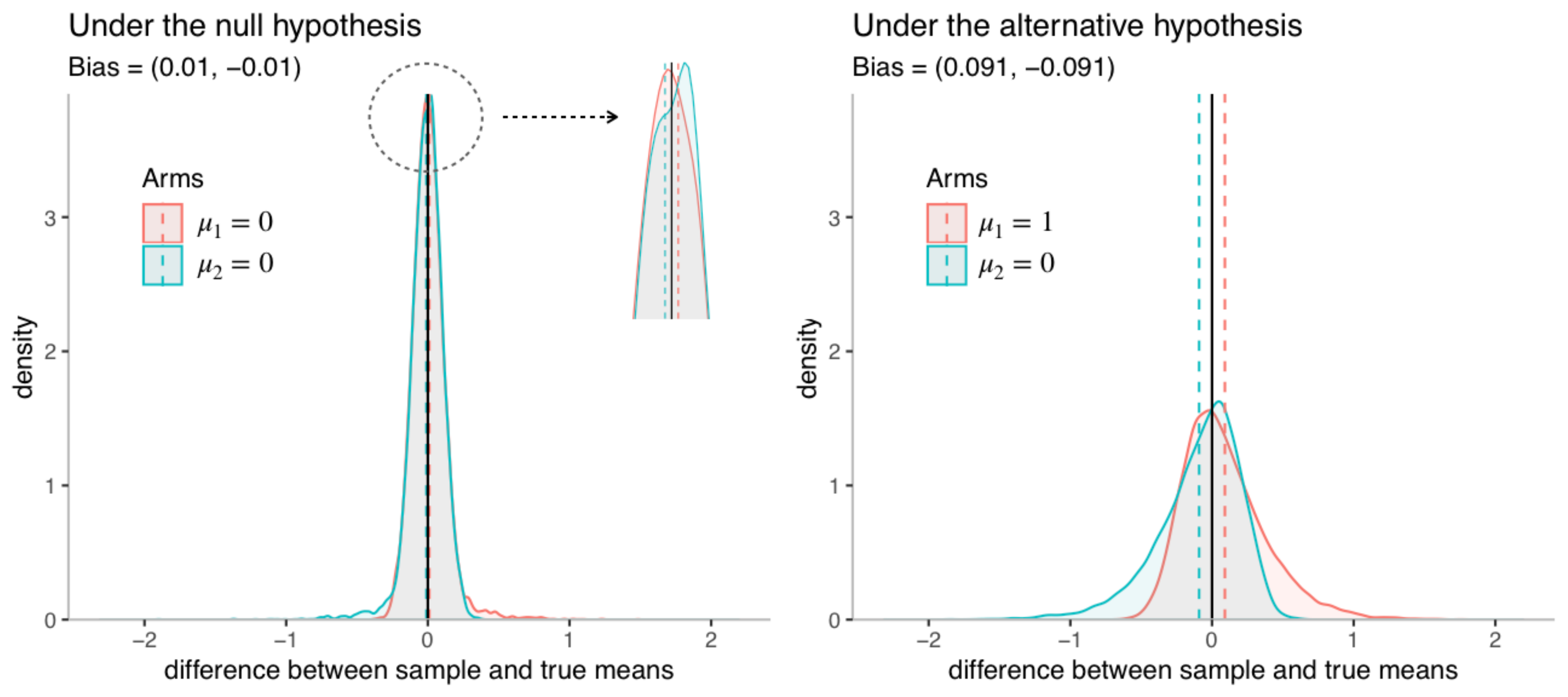}
	\end{center}
	\caption{\em  Data is collected from the one-sided sequential likelihood ratio test procedure described in Section~\ref{subSec::simulation_2}. The sample mean for arm $1$ is positively biased and the sample mean for arm $2$ is negatively biased under both null and alternative hypothesis cases. Note that the size of the bias under the null hypothesis is smaller than the one under the alternative hypothesis since the number of collected samples is larger under the null hypothesis. }
	\label{fig::simul_2}
\end{figure}

% \textcolor{blue}{[this paragraph can be moved to discussion]}

\subsection{Positive bias of the lil'UCB algorithm in best-arm identification}
\label{subSec::simulation_4}

Suppose we have $K$ sub-Gaussian arms with mean $\mu_1, \dots, \mu_K$ and known parameter $\sigma$. In the best-arm identification problem, our target of inference is the arm with the largest mean. There exist many algorithms for this task including lil'UCB  \citep{jamieson_lil_2014}, Top-Two Thompson Sampling \citep{russo2016simple} and Track-and-Stop \citep{garivier2016optimal}. 

In Corollary~\ref{cor::lil'UCB}, we showed that the lil'UCB algorithm is monotonically increasing, and thus the sample mean of the chosen arm is positively biased. In this subsection, we verify it with a simulation. It is an interesting open question whether different types of best-arm identification algorithms also yield positively biased sample means. 

The lil'UCB algorithm consists of the following optimistic sampling, stopping and choosing:
\begin{itemize}
    \item Sampling: For any $k \in [K]$ and $t = 1,\dots K$, define $\nu_t(k) = \mathbbm{1}(t = k)$. For $t > K$,
		\[
	\nu_t(k) = \begin{cases}
	1 & \text{if } k = \argmax_{j \in [K]} \hat{\mu}_j(t-1) + u_t^{\text{lil}}\left(N_j(t-1)\right), \\
	0 &  \text{otherwise, }
	\end{cases}
	\]
		where $\delta, \epsilon, \lambda$ and $\beta$ are algorithm parameters and  
		\[
		u_t^{\text{lil}}\left(n\right):= (1+\beta)(1+\sqrt{\epsilon})\sqrt{2\sigma^2(1+\epsilon)\log\left(\log((1+\epsilon)n) /\delta\right)/n}.
		\]
    \item Stopping: $\Tau = \inf\left\{t > K : N_k(t) \geq 1 + \lambda\sum_{j\neq k}N_j(t) \text{ for some }k \in [K] \right\}$.
    \item Choosing: $\kappa = \argmax_{k \in [K]} N_k(\Tau)$.
\end{itemize}

Once we stop sampling at time $\Tau$, the lil'UCB algorithm guarantees that $\kappa$ is the index of the arm with largest mean with some probability depending on input parameters. Based on this, we can also estimate the largest mean by the chosen stopped sample mean $\hat{\mu}_{\kappa}\left(\Tau\right)$. 
The performance of this sequential procedure can vary based on underlying distribution of the arm and the choice of parameters. However, 
we can check this optimistic sampling  and optimistic stopping/choosing rules which would yield negative and positive biases respectively are monotonic increasing and thus the chosen stopped sample mean $\hat{\mu}_{\kappa}\left(\Tau\right)$ is always positively biased for any choice of parameters. 

To verify it with a simulation, we set $3$ unit-variance Gaussian arms with means $(\mu_1, \mu_2, \mu_3) = (g, 0, -g)$ for each gap parameter $g = 1, 3, 5$. We conduct $10^4$ trials of the lil'UCB algorithm with a valid choice of parameters described in \citet[Section~5]{jamieson_lil_2014}. Figure~\ref{fig::simul_4} shows the distribution of observed differences between the chosen sample means and the corresponding true mean for each $\delta$. Vertical lines correspond to biases. The simulation study demonstrates that, in all configurations, the chosen stopped sample mean $\hat{\mu}_{\kappa}\left(\Tau \right)$ is always positively biased. (see Appendix~\ref{Appen::subSec::lilUCB_proof} for a formal proof.)
% Note that the bias is larger for a larger gap since the number of collected samples smaller on an easier task.

\begin{figure}[h!]
	\begin{center}
	\includegraphics[scale =  0.2]{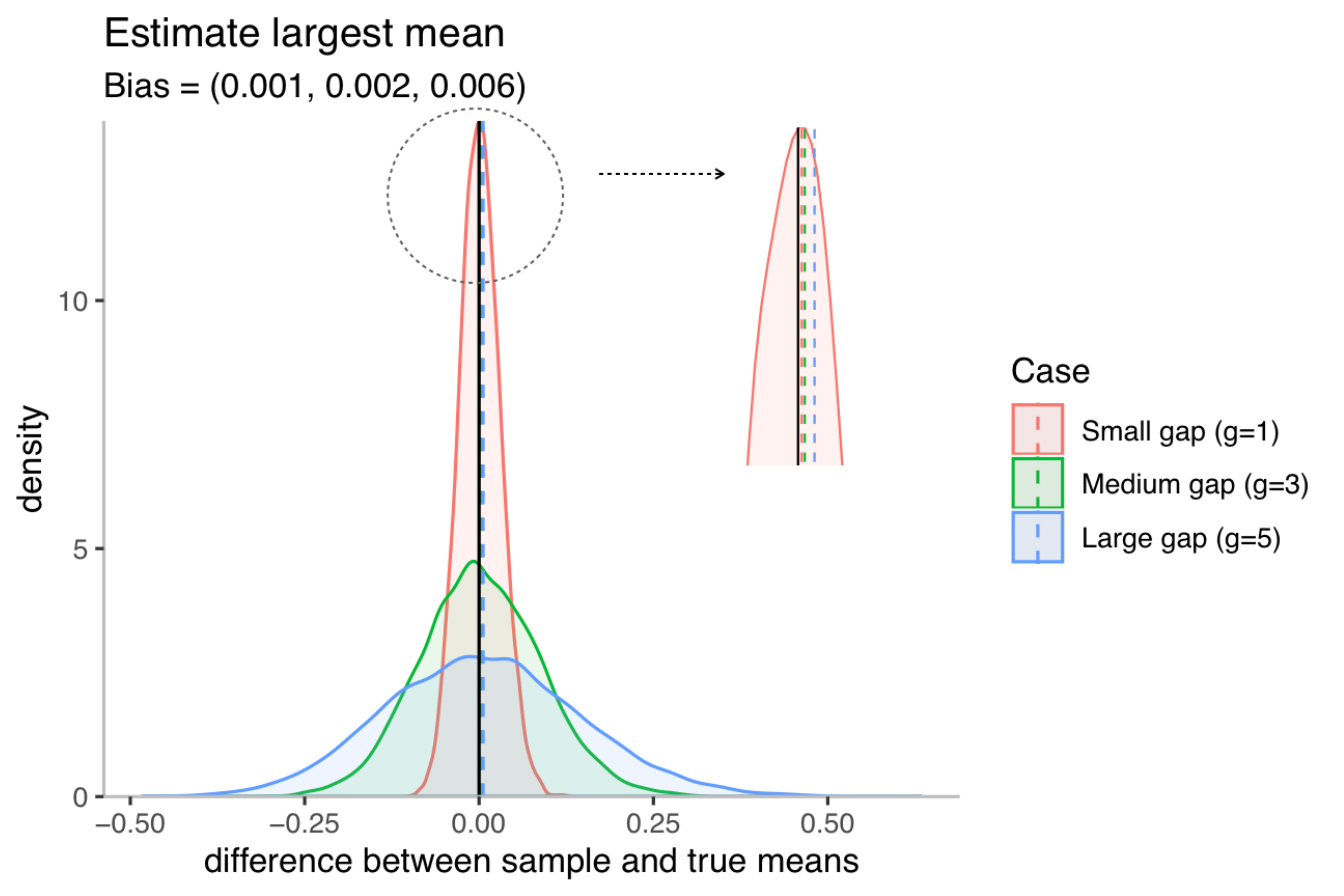}
	\end{center}
	\caption{\em  Data is collected by the lil'UCB algorithm run on three unit-variance Gaussian arms with $\mu_1=g, \mu_2 = 0$ and $\mu_3 = -g$ for each gap parameter $g = 1, 3, 5$. For all cases, chosen sample means are positively biased. The bias is larger for a larger gap since the number of collected samples is smaller on an easier task.}
	\label{fig::simul_4}
\end{figure}

\section{Summary} \label{sec::summary}

This paper provides a general and comprehensive characterization of the sign of the bias of the sample mean in multi-armed bandits. Our main conceptual innovation was to define new weaker conditions (monotonicity and optimism) that capture a wide variety of practical settings in both the random walk (one-armed bandit) setting and the MAB setting. Using this, our main theoretical contribution, Theorem \ref{thm::bias_sign}, significantly generalizes the kinds of algorithms or rules for which we can mathematically determine the sign of the bias for any problem instance. Our simulations confirm the accuracy of our theoretical predictions for a variety of practical situations for which such sign characterizations were previously unknown. There are several natural followup directions: (a) extending results like Corollary~\ref{cor::lil'UCB} to other bandit algorithms, (b) extending all our results to hold for other functionals of the data like the sample variance, (c) characterizing the magnitude of the bias. We have recently made significant progress on the last question \citep{shin2019bias}, but the other two remain open. 

% \subsection{Comparison with \citet{nie2018adaptively}}
% \textcolor{red}{TODO::Prove Exploit condition implies optimistic sampling.}

% As discussed in Section~\ref{subSec::optimistic_rules}, \citet{nie2018adaptively} introduced Exploit and IIO conditions and proved that if the sampling strategy satisfied two conditions then the sample mean is negatively biased for any fixed arm and at any fixed time. 
% For a fixed arm and time, the optimistic sampling condition we introduced also implies the negative bias of the sample mean. In this subsection, we prove our optimistic sampling condition is strictly weaker than Exploit and IIO conditions in \citet{nie2018adaptively} as following. 

\bibliographystyle{plainnat}
\bibliography{bias_sample_mean}

\appendix
\input{bias_sample_mean_appendix.tex}
\end{document}

%% file: bias_sample_mean_appendix.tex
\newpage
\section{\texorpdfstring{$\epsilon$}{eps}-greedy, UCB and Thompson sampling are optimistic sampling rules} \label{Appen::sampling}

\subsection{Descriptions of \texorpdfstring{$\epsilon$}{eps}-greedy, UCB and Thompson sampling rules} \label{Appen::subSec::description}
$\epsilon$-greedy, UCB, and Thomson sampling have the following sampling rules.
\begin{itemize}
	\item $\epsilon$-greedy algorithm : For any $k\in [K]$ and $t \in [\Tau]$,
	\[
	\nu_t(k) = \begin{cases}
	1 - \epsilon & \text{if } k = \argmax_{j \in [K]} \hat{\mu}_j(t-1),\\
	\frac{\epsilon}{K-1} &  \text{otherwise. }
	\end{cases}
	\]
	
	\item UCB : For any $k \in [K]$ and $t \in [\Tau]$,
		\[
	\nu_t(k) = \begin{cases}
	1 & \text{if } k = \argmax_{j \in [K]} \hat{\mu}_j(t-1) +u_{t-1}(S_j(t-1), N_j(t-1)),\\
	0 &  \text{otherwise, }
	\end{cases}
	\]
	 where $(s,n) \mapsto u_{t-1}(s,n)$ is a non-negative function which is increasing and decreasing with respect to the first and second inputs respectively for each $t$. For example, a simple version of UCB uses $u_{t-1}(s, n) = \sqrt{\frac{2\log(1/\delta)}{n}}$ for a properly chosen constant $\delta \in (0,1)$. 
	\item Thompson sampling :  For any $k \in [K]$ and $t \in [\Tau]$,
	\[
	\nu_t(k) \propto \pi(k = \argmax_j \mu_j \mid A_1, Y_1, \dots, A_{t-1}, Y_{t-1}).
	\]
	where $\pi$ is a prior on  $(\mu_1, \dots, \mu_K)$ or, more generally, on parameters of arms $(\theta_1, \dots, \theta_K)$. In particular, if underlying arms are Gaussian with common variance $\sigma^2$ and if we impose independent Gaussian prior $N(\mu_{k,0}, \sigma_{0}^2)$ on each arm $k$, the corresponding Thompson sampling is statistically equivalent to the following rule.
	\[
	\nu_t(k) = \begin{cases}
	1 & \text{if } k = \argmax_{j \in [K]} \tilde{\mu}_j(t-1) + \sigma_j(t-1) Z_{j, t-1}\\
	0 &  \text{otherwise, }
	\end{cases}
	\]
	where each $Z_{j, t-1}$ is an independent draw from $N(0,1)$ and $\tilde{\mu}_j(t-1), \sigma_k(t-1)$ are the posterior mean and standard deviation of arm $j$, given as
	\[
	\tilde{\mu}_j(t-1) = \frac{\mu_{j,0}/\sigma_0^2 + N_j(t-1) \hat{\mu}_j(t-1)/\sigma^2}{1/\sigma_0^2 + N_j(t-1)/\sigma^2 },
	~~\sigma_j(t-1) = \left(1/\sigma_0^2 + N_j(t-1)/\sigma^2\right)^{-1/2}.
	\]
\end{itemize}

\subsection{Exploit and IIO conditions are sufficient for optimistic sampling} \label{Appen::subSec::proof_claim}

%\begin{proof}[Proof of Claim~\ref{claim::exploit_imply_optimistic}]
%For any fixed arm $k$ and $T$, it is enough to show that under Exploit condition, the function $\D_\infty^* \mapsto  N_k(T)$ is an increasing function of $X_{i,k}^*$ for any $t \leq T$ while keeping all other entries in $\D_\infty^*$ fixed. 

In Fact~\ref{claim::exploit_imply_optimistic}, we claimed that  ``Exploit'' and ``IIO'' conditions in \citet{nie2018adaptively} are jointly a sufficient condition for a sampling rule being optimistic. In this subsection, we formally restate Exploit and IIO conditions of \cite{nie2018adaptively} in terms of our notations and prove Fact~\ref{claim::exploit_imply_optimistic}.

First, fix a deterministic stopping time $T$. Given any $t \in [T], k \in [K]$, define respectively the data from arm $k$ until time $t$, and the data from all arms except $k$ until time $t$, as
\[
\D_t^{(k)} := \left\{X_{i,k}^*\right\}_{i=1}^{N_k(t)} ~~\text{and}~~ \D_t^{(-k)} := \D_t \setminus \D_t^{(k)} = \bigcup_{j \neq k} \left\{X_{i,j}^*\right\}_{i=1}^{N_j(t)} \cup \{W_{-1}, W_0, \dots, W_{t}\}, 
\]
where $\D_t$ is the sample history up to time $t$ under a tabular model $\D_\infty^*$. Let $\D_\infty^{*'}$ be another tabular model. Under $\D_\infty^{*'}$, we define $\D'_t, {\D'}_t^{(k)}$ and ${\D'}_t^{(-k)}$ in the same way.  The Exploit condition in \citet{nie2018adaptively} can be rewritten as following.

\begin{definition}[Exploit]
Given any $t \in [T], k \in [K]$, suppose $\D_t^{(k)}$ and ${\D'}_t^{(k)}$ have the same size (that is $N_k'(t) = N_k(t)$) and $ \D_t^{(-k)} = {\D'}_t^{(-k)}$. If the sample mean $\hat{\mu}_k(t)$ under $\D_t^{(k)}$ is less than or equal to the sample mean $\hat{\mu}_k'(t)$ under ${\D'_t}^{(k)}$, then 
\[
\mathbbm{1}(A_t = k) := f_{t,k}\left(\D_t^{(k)} \cup \D_t^{(-k)}\right) \leq 
f_{t,k}\left({\D'}_t^{(k)} \cup \D_t^{(-k)}\right) =: \mathbbm{1}(A_t' = k). 
\]
\end{definition}

For the IIO condition, we present a specific version in the MAB setting which was originally used in Eq.(8) in the proof of Theorem 1 in \citet{nie2018adaptively}.

\begin{definition}[Independence of Irrelevant Options (IIO)]
 For each $t, k$, the sampling random variable $A_t$ can be written in terms of deterministic functions $f_{t,k}$ and $g_{t,k}$ such that
 	\[
	A_t = \begin{cases}
	k & \text{if } f_{t,k}\left(\D_{t-1}\right) = 1\\
	j &  \text{if } f_{t,k}\left(\D_{t-1}\right) = 0~~\text{and}~~g_{t,k}\left(\D_{t-1}^{(-k)}\right) = j~~\text{for some}~j \neq k. 
	\end{cases}
	\]
\end{definition}
Intuitively, $f_{t,k}$ is simply the indicator of whether arm $k$ was pulled at time $t$; the crucial part is $g_{t,k}$, which specifies which arm is selected when arm $k$ is not, and the IIO condition requires that $g_{t,k}$ ignores the data from arm $k$ in order to determine which $j \neq k$ to pull instead.

It can be checked that $\epsilon$-greedy, UCB and Thompson sampling under Gaussian arms and Gaussian priors satisfy both conditions. Indeed, if arm $k$ is not the arm with the highest mean or highest UCB (for example), determining which other arm does get pulled in the next step does not depend on the data from arm $k$. In Appendix~\ref{Appen::subSec::suff_cond_TS}, we present a sufficient condition for Thompson sampling to satisfy both conditions, and thus to be optimistic  which shows Thompson sampling is optimistic for many commonly used exponential family arms including Gaussian, Bernoulli, exponential and Possion arms with their conjugate priors. 

Before we prove Fact~\ref{claim::exploit_imply_optimistic}, we first introduce a lemma related to the IIO condition as follows.
\begin{lemma} \label{lemma::cor_of_IIO} 
Fix a $k \in [K]$. Let $\D_\infty^*$ and $\D_\infty^{*'}$ be two MAB tabular representation that  agree with  each other  except in their $k$-th column. Let $N_j(t)$ and $N_j'(t)$ be the numbers of draws from arm $j$ for all $j \in [K]$ under $\D_\infty^*$ and $\D_\infty^{*'}$ respectively. Then, under IIO, the following implication holds:
\begin{equation}
N_k(t) \leq N_k'(t) \Rightarrow N_j(t) \geq N_j'(t),    \quad \text{ for all } j \neq k.
\end{equation}
By switching the roles of  $\D_\infty^*$ and $\D_\infty^{*'}$, we also have
\begin{equation}
N_k(t) \geq N_k'(t) \Rightarrow N_j(t) \leq N_j'(t),  \quad  \text{ for all } j \neq k,
\end{equation}
 and therefore,
\begin{equation}
N_k(t) = N_k'(t) \Rightarrow N_j(t) = N_j'(t), \quad \text{ for all } j \neq k.
\end{equation}
\end{lemma}
\begin{proof}[Proof of Lemma~\ref{lemma::cor_of_IIO}] It is enough to prove the first statement. We follow the logic in the proof of Property~$1$ in \citet{nie2018adaptively}. If $N_k(t) = t$ or $N_k'(t) = t$ then the claimed statement holds trivially since $N_j(t) + N_k(t) \leq t$ and $N_j'(t) + N_k'(t) \leq t$ for all $j \neq k$. Therefore, for the rest of the proof, we assume $N_k(t) \leq N_k'(t) < t$.

For each $t$, define $s_1 <\dots <s_{t-N_k(t)}$ to be the sequence of times at which arm $k$ was \emph{not} sampled before time $t$ under $\D_\infty^*$. Similarly, let $s_1' <\dots <s_{t-N_k(t)}'$ be the sequence of times at which arm $k$ was \emph{not} sampled before time $t$ under $\D_\infty^{*'}$. From the IIO condition and the assumption that $\D_\infty^*$ and $\D_\infty^{*'}$ agree with  each other  except in their $k$-th column, we have 
\begin{equation} \label{sublem::cor_of_IIO}
A_{s_u} = A_{s_u'},~~\text{ for all } u \in \{1,\dots, t-N_k'(t)\},  
\end{equation}
which implies that
\[
N_j'(t) = N_j'(s'_{t-N_k'(t)}) = N_j(s_{t-N_k'(t)}) \leq N_j(s_{t-N_k(t)}) = N_j(t),
\]
where the first and the last identities stem from the definition of $s$ and $s'$, the second identity is due to  \eqref{sublem::cor_of_IIO}, and the inequality follows from the assumption that $N_k(t) \leq N_k'(t)$ along with the fact that $u \mapsto s_u$ and $s \mapsto N_j(s)$ are increasing.
\end{proof}

\begin{proof}[Proof of Fact~\ref{claim::exploit_imply_optimistic}]
Let us fix an arm $k$ and a deterministic stopping time $T$, and a time $t \leq T$, as required by Exploit and IIO conditions. The arguments below are  inspired by case~1 in the proof of Theorem~1 in \citet{nie2018adaptively}.

Let $X_{i,k}^{*'}$ be an independent copy of $X_{i,k}^*$ and define $X_{\infty}^{*'}$ as a $\mathbb{N} \times K$ table which equals $X_{\infty}^{*}$ on all entries except the $(i,k)$-th entry, which contains $X_{i,k}^{*'}$. Let $\D_{\infty}^{*'} =  X_{\infty}^{*'} \cup \{W_{-1}, W_0, \dots\}$  denote the corresponding dataset, which only differs from $\D^*_\infty$ in one element. Let $N_k(T)$ and $N_k'(T)$ be numbers of draws from arm $k$ up to time $T$ based on $\D_\infty^*$ and $\D_\infty^{*'}$ respectively. Also for each $t \leq T$, let $A_t$ and $A_t'$ be sampled arms based on $\D_\infty^*$ and $\D_\infty^{*'}$ respectively. 

To prove the claim, it is enough to show that if $X_{i,k}^* \leq X_{i,k}^{*'} $ then $N_k(T) \leq N_k'(T)$ under Exploit and IIO conditions. Suppose, for the sake of deriving a contradiction, that there exist $i \in \mathbb{N}$ and $k \in [K]$ such that $X_{i,k}^* \leq X_{i,k}^{*'} $ but $N_k(T) > N_k'(T)$. Note that since $A_s$ and $A_s'$ are functions of the history up to time $s-1$, we know that $A_s = A_s'$ for all $s \leq t$, where $t$ is defined as $t = \min\left\{s \geq 1 : N_k(s) = N_k'(s) = i\right\}$. If $t \geq T$, we have that $N_k(T) = N_k(t) = N_k'(t) =N_k'(T)$, which contradicts our assumption. Hence, we may assume $t < T$ for the rest of the proof.

Define $s_0 := \min\left\{s \geq 1 : N_k(s) > N_k'(s)\right\}$. From the definition of $s_0$, we know that $N_k(s_0-1) = N_k'(s_0-1)$. Since $\D_\infty^*$ and $\D_\infty^{*'}$ are identical except for their $(i,k)$-th entry, by Lemma~\ref{lemma::cor_of_IIO}, we have that $N_j(s_0-1) = N_j'(s_0-1)$ for all $j$, which also implies that $\D_{s_0-1}$ and $\D_{s_0 -1}'$ are identical except for the $N_k(t)$-th observation from arm $k$. Therefore, the sample mean from arm $k$ up to time $s_0 - 1$ under $\D_{s_0-1}'$ is larger than the one under $\D_{s_0-1}$.

 Then, by the Exploit condition, $A_{s_0} = k$ implies that $A_{s_0}'=k$. This contradicts the assumption that $N_k(s_0) > N_k'(s_0)$. Therefore, if $X_{i,k}^* \leq X_{i,k}^{*'}$ then $N_k(T)$ must be less than or equal to $N_k'(T)$. Since it holds for any $i \in \mathbb{N}$, $k \in [K]$ and $T$, the sampling strategy is optimistic, proving our claim that Exploit and IIO conditions are jointly a special case of an optimistic sampling rule.
\end{proof}

\subsection{Sufficient conditions for Thompson sampling to be optimistic}\label{Appen::subSec::suff_cond_TS}

In the previous subsection~\ref{Appen::subSec::proof_claim}, we show that Exploit and IIO conditions are jointly a sufficient condition for a sampling rule to be optimistic. In this subsection, we present a sufficient condition for Thompson sampling to satisfy both conditions, and thus to be optimistic.

For each $k$, let $\theta_k$ be the parameter of the distribution of arm $k$, and let $\mu_k = \mu(\theta_k)$. If we use an independent prior $\pi$ on $\theta := (\theta_1, \dots, \theta_K)$, it can be easily shown that posterior distributions of $\theta$ and $\mu(\theta) := (\mu(\theta_1), \dots, \mu(\theta_K))$ are also coordinate-wise independent conditionally on the data. Therefore, the IIO condition is trivially satisfied for the Thompson sampling algorithms. However, it is difficult to check whether the Exploit condition is satisfied because there is no closed form for $\pi(k = \argmax_{j\in [K]} \mu(\theta_j) | \D_t)$ in general. 

One way to detour this issue is to study whether there exists a posterior sampling method such that the following statistically equivalent sampling algorithm satisfies the Exploit condition.
	\[
	\nu_t(k) = \begin{cases}
	1 & \text{if } k = \argmax_{j \in [K]} \mu_j(\theta_{j,t-1})\\
	0 &  \text{otherwise, }
	\end{cases}
	\]
	where $\theta_{j, t-1}$ is a draw from the posterior distribution $\pi(\theta_j | \D_{t-1})$ at time $t-1$. If there exists such sampling method, we know that the sample mean from this Thompson sampling is negatively biased for any fixed $k$ and $T$. With a slight abuse of notation, we say the Thompson sampling is optimistic in this case. 
	
	For example, in Appendix~\ref{Appen::subSec::proof_claim}, we show that Thompson sampling under Gaussian arm and Gaussian prior is optimistic by using a standard Gaussian posterior sampling method described in Appendix~\ref{Appen::subSec::description}. Similarly, for the 
	Bernoulli arm with parameters $\{p_k\}_{k=1}^K$ and beta prior with non-negative integer parameters $(n,m)$ case, we can check that the corresponding Thompson sampling is optimistic using the equivalent  optimistic sampling rule
	\[
	\nu_t(k) = \begin{cases}
	1 & \text{if } k = \argmax_{j \in [K]} \frac{a_{j, t-1}}{a_{j, t-1} + b_{j, t-1}}\\
	0 &  \text{otherwise, }
	\end{cases}
	\]
	where 	$a_{j, t-1} = -\sum_{i=1}^{n + S_k(t-1)} \log U_{i,k}$, $b_{j, t-1} = -\sum_{i=1}^{m + N_k(t-1) - S_k(t-1)} \log W_{i,k}$ and each $U_{i,k}$ and $W_{i,k}$ are independent draws from $U(0,1)$. 

    In general, we have the following sufficient condition for the Thompson sampling to be optimistic. 
    
    \begin{corollary} \label{cor::sufficient_TS}
    Suppose the distributions of the arms belong to a one-dimensional exponential family with density $p_\eta (x) = \exp\{\eta T(x) - A(\eta)\}$ with respect to some dominating measure $\lambda$ and with $\eta \in \mathrm{E}$. Let $\pi$ be a conjugate prior on $\eta$ with a density proportional to $\exp\{\tau \eta \ - n_0A(\eta)\}$. If $\pi(\eta \leq x \mid \tau, n_0)$ is a decreasing function of $\tau$ for any given $x$ and $n_0$, and if $\eta \mapsto \mu(\eta)$ and $x \mapsto T(x)$ are both increasing or decreasing mappings, then Thompson sampling is optimistic.
    \end{corollary}
    \begin{proof}
    Fix a an arm $k \in [K]$. By the conjugacy, the posterior distribution for $\eta_k$ given the data up to time $t$ is given by 
    \[
    \pi\left(\eta_k | \D_{t}\right)\propto\exp\left\{\left(\tau + S^T_k(t)\right) \eta_k - \left(n_0 + N_k(t)\right)A(\eta_k)\right\},
    \]
    where $S^T_k(t) := \sum_{s = 1}^t \mathbbm{1}(A_s = k)T(Y_s)$.  Let $F\left(x | S^T_k(t), N_k(t)\right) := \pi\left(\eta_k \leq x | \D_{t}\right)$. From the condition on the prior, we know that $S_k^T(t)~\mapsto~F\left(x | S_k^T(t), N_k(t)\right)$ is a decreasing mapping for any given $x, N_k(t)$ and indices $i, k$ and $t$. Therefore $S_k^T(t)~\mapsto~F^{-1}\left(y | S_k^T(t), N_k(t)\right)$ is an increasing mapping for any given $y, N_k(t)$ and indices $i, k$ and $t$. Now, we can check that the Thompson sampling is equivalent to the following sampling rule.
    \[
	\nu_t(k) = \begin{cases}
	1 & \text{if } k = \argmax_{j \in [K]} \mu\left(\eta_{j, t-1}\right)\\
	0 &  \text{otherwise, }
	\end{cases}
	\]
	where $\eta_{j, t-1} := F^{-1}\left(U_{j, t-1} | S_k^T(t-1), N_k(t-1)\right)$ and each $U_{j, t-1}$ is an independent draw from $U(0,1)$. Since $\eta \mapsto \mu(\eta)$ and $x \mapsto T(x)$ are both increasing (or decreasing), this sampling rule and the corresponding Thompson sampling is optimistic.
    \end{proof}
    We can check many commonly used one-dimensional exponential family arms with its conjugate prior satisfying the condition in Corollary~\ref{cor::sufficient_TS} which includes Gaussian distributions with a Gaussian prior, Bernoulli distributions with a beta prior, Poisson distributions with a gamma prior and exponential distributions with a gamma prior

\subsection{Intuitions for the sign of the bias under each optimistic sampling and stopping}

\label{sec::intuition}
Under an optimistic sampling rule with a fixed stopping time and a fixed target, \citet{xu2013estimation} and \citet{nie2018adaptively} provided some intuitions as to why the sample mean is negatively biased. In this subsection, we presents a similar intuitive explanation for the negative bias of the sample mean due to adaptive sampling. We also offer some intuition in order to explain the positive bias stemming from optimistic stopping rules in the one-armed case.  

For an optimistic sampling rule with a fixed stopping time, assume for simplicity that we have a fixed target arm with a symmetric distribution around its true mean. Consider two equally possible realization of the experiment up to time $t$. In one realization, the sample mean at time $t$ happens to be larger than its true mean. On the other hand, in the other scenario, the sample mean at time $t$ happens to be smaller than its true mean. In the first case, the optimistic sampling rule will draw samples more often from the target arm, and thus the sample mean will regress more easily to its true mean. In contrast, in the other case, the optimistic sampling rule will draw samples less often and thus the sample mean is less likely to regress to its true mean due to the smaller sample size. Since these two realizations are equally likely, on average, the sample mean is negatively biased. See Figure~\ref{fig::intution_sampling} for an illustration of this intuition.

For optimistic stopping in the one-armed case, consider the  stopping rule that terminates the experiment when the sample mean crosses a predetermined upper boundary. See Figure~\ref{fig::intution_stopping} for an illustrative stopping boundary. As we did for the sampling case, we again assume that the distribution of the arm is symmetric around its true mean. As before, consider two equally possible realizations. In one realization, the sample mean at early times happens to be larger than the true mean. On the other hand, in the other realization, the sample means at early times is smaller than its true mean. In the first realization, the sample mean will cross the upper stopping boundary at an earlier time and thus the sample mean at the crossing time will be large. In contrast, in the other realization, the sample mean will cross the boundary at a later time and thus the optimistic stopping rule ensures that we will draw more samples in this realization and thus the sample mean is more likely to regress to its true mean due to the larger sample size. Since these two realizations are equally likely, on average, the sample mean is positively biased. See Figure~\ref{fig::intution_stopping} for an illustration of this intuition.

\begin{figure}[!h]
\begin{center}
\includegraphics[scale=.35]{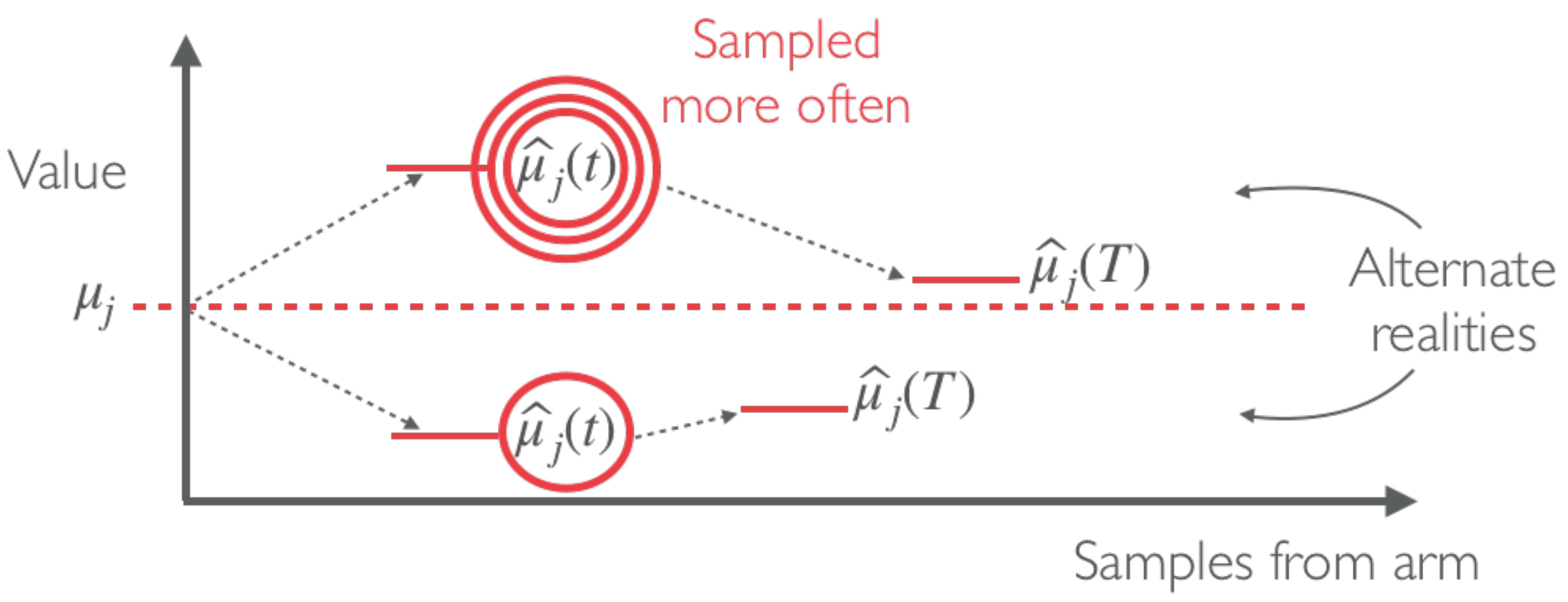}
\end{center}
	\caption{\em An illustration of the intuition for why optimistic sampling results in negative bias. }
	\label{fig::intution_sampling}
\end{figure}

\begin{figure}[!h]
\begin{center}
\includegraphics[scale=.35]{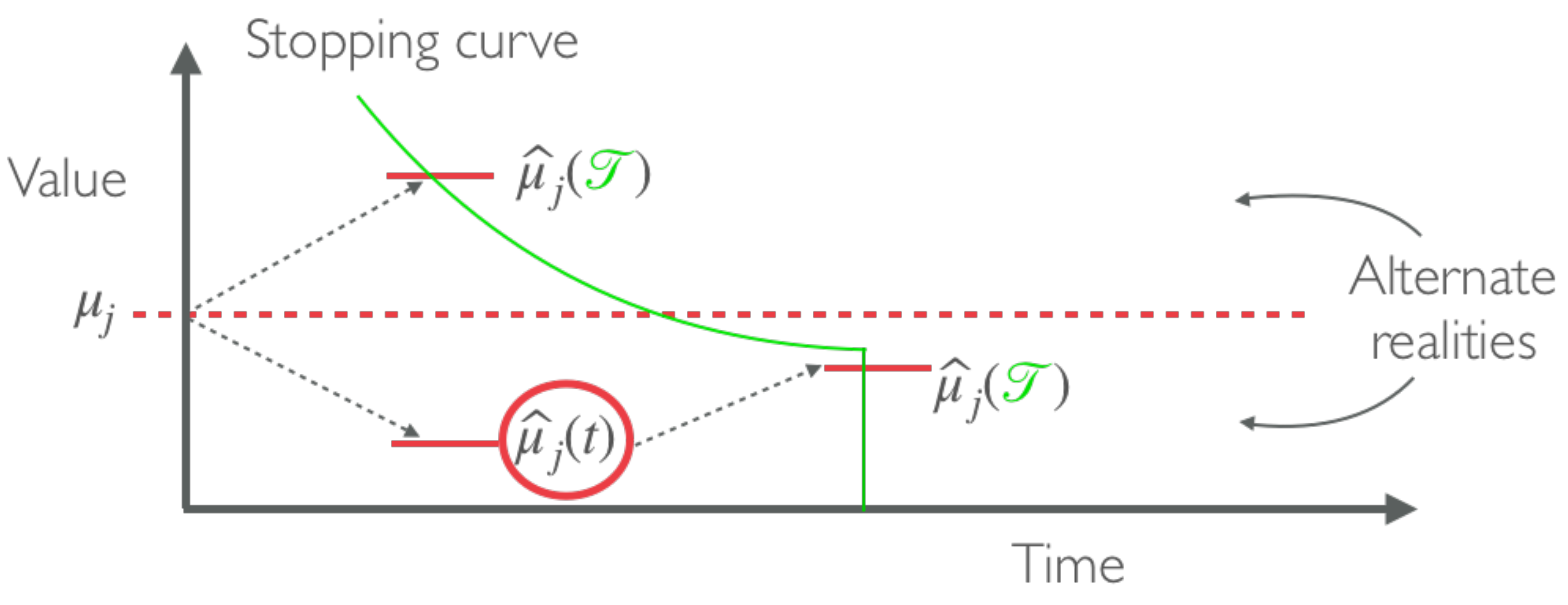}
\end{center}
	\caption{\em An illustration of the intuition for why optimistic stopping results in positive bias.  }
	\label{fig::intution_stopping}
\end{figure}

\section{Proofs}
\subsection{Proof of Theorem~\ref{thm::bias_sign} (the paper's central theorem on the sign of the bias)}
\label{sec:Proof_of_bias_sign}
% \textcolor{red}{JH : Entire Appendix B is rewritten, thus I skipped to put colors on the updated proof.}

Suppose that the data collecting strategy is monotonically decreasing for the $k$-th distribution. Then, we will first show that, for any time $t \in \mathbb{N}$, we have
 \begin{equation}\label{eq::monotone_increasing_ineq}
 	    \mathbb{E}\left[ \frac{\mathbbm{1}\left(\kappa = k\right)}{ N_k(\Tau) }\mathbbm{1}\left(A_t = k\right)\left(Y_t - \mu_k\right)  \mid \mathcal{F}_{t-1} \right]
	\leq 0.
 \end{equation}
  Similarly, if the data collecting strategy is monotonically increasing, the inequality is reversed. 
 It is understood that if $t > \Tau$, then $\mathbbm{1}\left(A_t = k\right)=0$ for all $k$, making the above claim trivially true, and hence below we implictly focus on $t \leq \Tau$.
%  for the $k$-th distribution, we have that 
%  \begin{equation}\label{eq::monotone_decreasing_ineq}
%  	    \mathbb{E}\left[ \frac{\mathbbm{1}\left(\kappa = k\right) }{ N_k(\Tau) }\mathbbm{1}\left(A_t = k\right)\left(Y_t - \mu_k\right)  \mid \mathcal{F}_{t-1} \right]
% 	\geq 0.
%  \end{equation}

\begin{proof}[Proof of inequality~\eqref{eq::monotone_increasing_ineq}]
Note that the LHS of inequality~\eqref{eq::monotone_increasing_ineq} can be rewritten as
 \begin{align*}
 	 &\mathbb{E}\left[ \frac{\mathbbm{1}\left(\kappa = k\right)}{ N_k(\Tau) }\mathbbm{1}\left(A_t = k\right)\left(Y_t - \mu_k\right)  \mid \mathcal{F}_{t-1} \right]\\
 	 &=\mathbb{E}\left[ \frac{\mathbbm{1}\left(\kappa = k\right)}{ N_k(\Tau) }\mathbbm{1}\left(A_t = k\right)\left(X_{N_k(t), k} - \mu_k\right)  \mid \mathcal{F}_{t-1} \right]\\
	&=  \mathbb{E}\left[\sum_{i =1}^t \frac{\mathbbm{1}\left(\kappa = k\right)}{ N_k(\Tau) }\mathbbm{1}\left(A_t = k\right)\mathbbm{1}\left(N_k(t) = i \right)\left(X_{i,k}^* - \mu_k\right)  \mid \mathcal{F}_{t-1} \right] \\
	& = \sum_{i =1}^t \mathbb{E}\left[ \frac{\mathbbm{1}\left(\kappa = k\right)}{ N_k(\Tau) }\mathbbm{1}\left(A_t = k\right)\mathbbm{1}\left(N_k(t) = i \right)\left(X_{i,k}^* - \mu_k\right)  \mid \mathcal{F}_{t-1} \right].
	\end{align*}
Therefore, it is enough to show the following inequality holds:
\begin{equation} \label{eq::WTS_monotone_ineq}
     \mathbb{E}\left[ \frac{\mathbbm{1}\left(\kappa = k\right)}{ N_k(\Tau) }\mathbbm{1}\left(A_t = k\right)\mathbbm{1}\left(N_k(t) = i \right)\left(X_{i,k}^* - \mu_k\right)  \mid \mathcal{F}_{t-1} \right]
	\leq 0,
\end{equation}
for each $t, k, i$. Recall that $\D_{\infty}^* = X_{\infty}^* \cup \{W_{-1}, W_0, \dots\}$ is a hypothetical dataset containing all possible independent samples from the distributions, external random sources and random seeds where the $(i,k)$-th entry of the table $X_{\infty}^*$ is a draw $X_{i,k}^*$  from $P_k$ independent of every other entry of $ X_{\infty}^*$ and of the external random sources and the random seeds $\{W_{-1}, W_0, \dots\}$. Let $X_{i,k}^{*'}$ be an independent copy of $X_{i,k}^*$ and define $X_{\infty}^{*'}$ as a $\mathbb{N} \times K$ table that equals $X_{\infty}^{*}$ on all entries except the $(i,k)$-th entry, which contains $X_{i,k}^{*'}$. Let $\D_{\infty}^{*'} =  X_{\infty}^{*'} \cup \{W_{-1}, W_0, \dots\}$  denote the corresponding dataset, which only differs from $\D^*_\infty$ in one element. Note that, for each $t$, we have
\begin{align*}
\mathcal{F}_{t} &= \sigma\left(\{W_{-1}, W_{0}, Y_1, W_1, \dots, Y_t, W_t\}\right)\\
&= \sigma\left( \bigcup_{k=1}^K \{X_{N_k(s),k}^*\}_{s=1}^t \cup \{W_{s}\}_{s=-1}^{t}\right),   
%& = \sigma\left( \bigcup_{k=1}^K  \bigcup_{i=1}^\infty  \{N_k(s), \mathbbm{1}\left(N_k(s) = i\right)X_{i ,k}^*\}_{s = 1}^t \cup \{W_{s}\}_{s=-1}^{t-1}\right),
\end{align*}
because there is an one-to-one correspondence between sets of random variables generating $\sigma$-algebras. Therefore $X_{i,k}^{*'}$ is independent to $\mathcal{F}_t$ for any choice of $i,k$ and $t$.

For any $i, k$ and $t$, since $\mathbbm{1}\left(A_t = k\right)$  and  $\mathbbm{1}\left(N_k(t) = i\right)$ are not functions of either $X_{i,k}^*$ or $X_{i,k}^{*'}$, if the data collecting strategy is monotonically decreasing, we have that
%\begin{equation} \label{eq::association_ineq_basic}
\[
      \mathbbm{1}\left(A_t = k\right) \mathbbm{1}\left(N_k(t) = i\right)\left(\frac{g_k(\D_{\infty}^*)}{f_k(\D_{\infty}^*) } - \frac{g_k(\D_{\infty}^{*'
    })}{f_k(\D_{\infty}^{*'})} \right)\left(\left(X_{i,k}^{*} - \mu_k \right)  -   \left(X_{i,k}^{*'} - \mu_k \right) \right)  \leq 0 .
\]
    Rearranging, we obtain that
    \begin{align*}
    &\mathbbm{1}\left(A_t = k\right)\mathbbm{1}\left(N_k(t) = i\right) \left(\frac{g_k(\D_{\infty}^*)}{f_k(\D_{\infty}^*)} \left(X_{i,k}^{*} - \mu_k\right) + \frac{g_k(\D_{\infty}^{*'}) }{f_k(\D_{\infty}^{*'}) } \left(X_{i,k}^{*'}-\mu_k\right)\right)\\
    & \leq \mathbbm{1}\left(A_t = k\right)\mathbbm{1}\left(N_k(t) = i\right)\left(\frac{g_k(\D_{\infty}^{*'}) }{f_k(\D_{\infty}^{*'}) }\left( X_{i,k}^{*} - \mu_k\right) + \frac{g_k(\D_{\infty}^*)}{f_k(\D_{\infty}^*) }\left(X_{i,k}^{*'} - \mu_k \right)\right) .
\end{align*}
%\end{equation}
% Now, let $N_k'(\Tau) := \sum_{s\geq 1} f_{t,k}(\D^{*'}_{s-1})  $ be the number of draws from $k$-th distribution at stopping time under $\D_s^{*'}$. Then, we can re-write the inequality~\eqref{eq::association_ineq_basic} as
% \[
% 0\geq \left(\frac{1}{N_k(\Tau) } - \frac{1}{N_k'(\Tau)} \right) \left(  \mathbbm{1}(A_t = k) X_{i,k}^{*}  -  \mathbbm{1}(A_t = k) X_{i,k}^{*}' \right). 
% \]

% Given the sample history up to time $t-1$, $\left(N_k(\Tau), X_{i,k}^{*}\right)$ and $\left(N_k'(\Tau), X_{i,k}^{*}'\right)$ have the same law. Also, we have  $N_k(\Tau) \indep X_{i,k}^{*}'$ and $N_k'(\Tau) \indep X_{i,k}^{*}$. Therefore,
Next, note that $\frac{g_k(\D_{\infty}^*)}{f_k(\D_{\infty}^*)} \left(X_{i,k}^{*} - \mu_k\right)$ and $\frac{g_k(\D_{\infty}^{*'}) }{f_k(\D_{\infty}^{*'}) } \left(X_{i,k}^{*'}-\mu_k\right)$ have the same  distribution and so do $\frac{g_k(\D_{\infty}^{*'}) }{f_k(\D_{\infty}^{*'}) } \left( X_{i,k}^{*} - \mu_k\right)$ and $\frac{g_k(\D_{\infty}^*)}{f_k(\D_{\infty}^*) } \left(X_{i,k}^{*'} - \mu_k \right)$. Therefore, by taking conditional expectation given $\mathcal{F}_{t-1}$ on both sides, we have
	\begin{align}
	    &2\mathbb{E}\left[\mathbbm{1}\left(A_t = k\right)\mathbbm{1}\left(N_k(t) = i\right)\frac{g_k(\D_{\infty}^*)}{f_k(\D_{\infty}^*)}  \left(X_{i,k}^{*} - \mu_k\right) \mid \mathcal{F}_{t-1}\right]\label{eq:ineq}  \\
	    &\leq 2\mathbb{E}\left[\mathbbm{1}\left(A_t = k\right)\mathbbm{1}\left(N_k(t) = i\right)\frac{g_k(\D_{\infty}^*)}{f_k(\D_{\infty}^*) } \left(X_{i,k}^{*'} - \mu_k \right) \mid \mathcal{F}_{t-1}\right] \nonumber \\
	    & = 2\mathbb{E}\left[\mathbbm{1}\left(A_t = k\right)\mathbbm{1}\left(N_k(t) = i\right)\frac{g_k(\D_{\infty}^*)}{f_k(\D_{\infty}^*)} \mid \mathcal{F}_{t-1}\right] \mathbb{E} \left[X_{i,k}^{*'} - \mu_k\right] \nonumber \\
	    & = 0 \nonumber,
	   % & = 2\mathbb{E}\left[\frac{1}{N_k(\Tau) } \mid \mathcal{F}_{t-1}\right] \mathbb{E} \left[\mathbbm{1}(A_t = k) \left(Y_{t} - \mu_k\right) \mid \mathcal{F}_{t-1} \right]. 
	\end{align}
where the first equality comes from the fact $X_{i,k}^{*'}$ is independent of both $\frac{g_k(\D_{\infty}^*)}{f_k(\D_{\infty}^*) }$ and $\mathcal{F}_{t-1}$ and that $\mathbbm{1}\left(A_t = k\right)$ and $\mathbbm{1}\left(N_k(t) = i\right)$ are measurable with respect to $\mathcal{F}_{t-1}$.
By plugging-in the identity $\frac{g_k(\D_{\infty}^*)}{f_k(\D_{\infty}^*)}= \frac{\mathbbm{1}\left(\kappa = k\right)}{N_k(\Tau)}$  into the LHS of \eqref{eq:ineq}, we obtain the inequality~\eqref{eq::WTS_monotone_ineq}, and thus, the inequality~\eqref{eq::monotone_increasing_ineq} as desired.
% The  inequality~\eqref{eq::monotone_decreasing_ineq} under the monotonically decreasing data collecting strategy for the $k$-th distribution can be shown by the same argument. 
\end{proof}

\begin{proof}[Proof of the signs of the covariance and  bias terms, equations  \eqref{eq::sign_of_cov_bias_increasing_each_k} and \eqref{eq::sign_of_cov_bias_decreasing_each_k}]
Suppose that the data collection strategy is monotonically increasing. Consider any arm $k$ such that $\mathbb{P}(\kappa = k) >0$. To prove equation~\eqref{eq::sign_of_cov_bias_increasing_each_k}, it is enough to show that $\mathbb{E}\left[\left(\hat{\mu}_{\kappa} - \mu_{\kappa}\right)\mathbbm{1}(\kappa = k)\right] \leq 0$. For each $t \geq 0$, define a process that is adapted to the natural filtration $\{\mathcal{F}_t\}_{t \geq 0}$ such that $L(0) = 0$ and
    \begin{equation}
	L(t) := \mathbb{E} \left[\frac{S_k(t) - \mu_k N_k(t)}{N_k (\Tau)} \mathbbm{1}(\kappa = k) \mid \mathcal{F}_t\right],~~\forall t \geq 1.
	\end{equation}
Note that the theorem requires us to show that $\E[L_\Tau] \leq 0$. We will first show that 
\begin{equation}\label{eq:super-mg}
\text{
$\{L(t)\}_{t \geq 0}$ is a super-martingale with respect to $\left\{\mathcal{F}_{t}\right\}_{t\geq 0}$.}
\end{equation}
First note that using inequality~\eqref{eq::monotone_increasing_ineq}, we have
	\begin{align*}
\mathbb{E}\left[L(1) \mid \mathcal{F}_0 \right] &= \mathbb{E}\left[ \frac{\mathbbm{1}\left(\kappa = k\right)}{N_k(\Tau)} \mathbbm{1}\left(A_1 = k\right)\left(Y_1 - \mu_k\right) \mid \mathcal{F}_0 \right]  \leq 0 = L(0).
	\end{align*}
Next, for all $t \geq 1$, again using inequality~\eqref{eq::monotone_increasing_ineq}, we have
	\begin{align*}
	\mathbb{E}\left[L(t) \mid \mathcal{F}_{t-1} \right] 
	& = L(t-1) +\mathbb{E}\left[ \frac{\mathbbm{1}\left(\kappa = k\right)}{N_k(\Tau)}\mathbbm{1}\left(A_t = k\right) \left(Y_t - \mu_k\right) \mid \mathcal{F}_{t-1} \right]  \\
%     & \leq  L_k(t-1) + \mathbb{E}\left[ \frac{1 }{ N_k(\Tau) } \mid \mathcal{F}_{t-1} \right] \mathbb{E}\left[\mathbbm{1}\left(A_t = k\right)\left(Y_t - \mu_k\right) \mid \mathcal{F}_{t-1} \right]~~\text{(by Claim~\ref{claim::association_ineq}.)}\\
% 	&=  L_k(t-1) + \mathbb{E}\left[ \frac{1}{ N_k(\Tau)} \mid \mathcal{F}_{t-1} \right] \mathbb{E}\left[\mathbbm{1}\left(A_t = k\right) \left(Y_t - \mu_k\right) \right] ~~\text{(since $\mathbbm{1}\left(A_t = k\right) \in \mathcal{F}_{t-1}$.)} \\
	& \leq L(t-1).
	\end{align*}
(Note that since sampling stops at time $\Tau$, it is understood that for $t>\Tau$, we have $\mathbbm{1}(A_t=\kappa)=0$, $S_\kappa(t)=S_\kappa(\Tau)$, $N_{\kappa}(t)=N_\kappa(\Tau)$, $\F_t = \F_\Tau$ and thus $L(t)=L(t-1)=L(\Tau)$, so the above inequality is still valid.)
	This proves claim \eqref{eq:super-mg}. By the optional stopping theorem, we have that
	\[
	\mathbb{E}L(\Tau \wedge t) \leq \mathbb{E}[L(0)] = 0, ~~\forall t \geq 1. 
	\]
	To prove $\mathbb{E}L(\Tau) \leq \mathbb{E}[L(0)]$, we follow the standard proof technique for the optional stopping theorem. To be specific, it is enough show that   $|L(\Tau \wedge t)| \leq U$ for all $t \geq 0$, where $U$ is such that $\mathbb{E}[U]< \infty$. The  result then follows from the dominated convergence theorem. Define $U$ as
    \begin{equation}\label{eq::U}
    U =  \sum_{s=1}^{\Tau} \left|L(s) - L(s-1) \right| 
    =  \sum_{s=1}^{\infty} \left|L(s) - L(s-1) \right| \mathbbm{1}\left( \Tau \geq s  \right).
    \end{equation}
    Clearly, $|L(\Tau \wedge t)| \leq U$ for all $t$. In order to show that $\mathbb{E}[U] < \infty$,  first note that for any $t \geq 1$, we have
    \begin{equation} \label{eq::integrability_condition_for_S_k}
	\begin{aligned}
	\mathbb{E}\left[\left|L(t+1) - L(t) \right|\mid \F_t\right] 
	& =  \mathbb{E}\left[ \frac{\mathbbm{1}\left(\kappa = k\right)}{N_k(\Tau)}\mathbbm{1}\left(A_{t+1} = k\right) \left|Y_{t+1} - \mu_k\right| \mid \mathcal{F}_{t} \right]  \\
	&\leq \mathbb{E}\left[ \mathbbm{1}(A_{t+1} = k)  \left|Y_{t+1} -\mu_k\right| \mid \F_t\right] 	 \\
	&=   \mathbbm{1}(A_{t+1} = k)  \mathbb{E}\left[ \left|Y_{t+1} -\mu_k\right| \mid \F_t\right]	 \\
	&=\mathbbm{1}(A_{t+1} = k) \int |x-\mu_k| \mathrm{d}P_k(x) \\
	&:=  c_k  \mathbbm{1}(A_{t+1} = k) , 
	\end{aligned}	    
	\end{equation}
    where the first inequality comes from the assumption $N_k(\Tau) \geq 1$ for all $k$ with $\mathbb{P}(\kappa =k) > 0$, and the following equality holds because $\mathbbm{1}(A_{t+1} = k) \in \mathcal{F}_t$. The third equality stems from the observation that, on the event $(A_{t+1} = k)$, $Y_{t+1} \sim P_k$ and it is independent of the previous history.  Therefore, we obtain that 
    \begin{align*}
        \mathbb{E}[U] &=   \sum_{s=1}^{\infty} \mathbb{E}\left[\mathbb{E} \left[ \left|L(s) - L(s-1) \right| \mathbbm{1}\left( \Tau \geq s  \right) \mid \mathcal{F}_{s-1} \right]\right] \\
        & =    \sum_{s=1}^{\infty} \mathbb{E} \left[\mathbbm{1}\left( \Tau \geq s  \right) \mathbb{E} \left[ \left|L(s) - L(s-1) \right|  \mid \mathcal{F}_{s-1} \right]\right] ~~\text{(since $\mathbbm{1}\left( \Tau \geq s  \right) \in \mathcal{F}_{s-1}$.)}\\
         & \leq   c_k \sum_{s=1}^{\infty} \mathbb{E} \left[\mathbbm{1}\left(A_s = k \right)\mathbbm{1}\left( \Tau \geq s  \right)\right] ~~\text{(by the inequality~\eqref{eq::integrability_condition_for_S_k})}\\
        & =  c_k\mathbb{E} N_{k}(\Tau) < \infty,
    \end{align*}
      where the finiteness of the last term follows from the assumption $\mathbb{E} N_k(\Tau)  < \infty$ for all $k$ with $\mathbb{P}(\kappa = k) > 0$. %and the fact that the cumulant generating function exists in an open neighbor of zero, which implies that $\int |x| \mathrm{d}P_k(x) = c_k < \infty$. 
    By the dominated convergence theorem, we have that
	\begin{align*}
	\mathbb{E}\left[\hat{\mu}_\kappa(\Tau) - \mu_\kappa \mid \kappa = k \right] \mathbb{P}(\kappa = k) &= \mathbb{E}\left[\left(\hat{\mu}_\kappa(\Tau) - \mu_\kappa\right)\mathbbm{1}(\kappa = k) \right]\\
	&= \mathbb{E}\left[L(\Tau)\right] \leq \mathbb{E}[L(0)] = 0,
	\end{align*}
	which implies that $\mathbb{E}\left[\hat{\mu}_\kappa \mid \kappa = k\right] \leq \mu_k$. 
	The inequality \eqref{eq::sign_of_cov_bias_increasing} follows  immediately from this result and the identity 
	\[
		\mathbb{E}\left[\hat{\mu}_\kappa(\Tau) - \mu_\kappa \right]  = \sum_{k: \mathbb{P}(\kappa =k) >0}\mathbb{E}\left[\hat{\mu}_\kappa(\Tau) - \mu_\kappa \mid \kappa = k \right] \mathbb{P}(\kappa = k).
	\]
    Thus, the sample mean at the stopping time $\Tau$ is negatively biased. 
    
    If the data collecting strategy is monotonically increasing, the supermartingale is replaced by a submartingale and the inequalities are reversed. This observation completes the proof.
    % by  inequality~\eqref{eq::monotone_decreasing_ineq}, we can show that $\{L(t)\}_{t\geq 0}$ is a sub-martingale with respect to $\{\mathcal{F}_t\}_{t\geq 0}$. Thus, the same arguments used above yield that
    % 	\[
    % \mathbb{E}\left[\hat{\mu}_\kappa(\Tau) - \mu_\kappa \right] =  \mathbb{E} L(\Tau) \geq \mathbb{E}L_k(0) =0,
    % \]
    % as desired.
    
    Now, suppose each arm has a bounded distribution. without loss of generality, assume there exists a fixed $M > 0$ such that $P_k\left([\mu_k-M, \mu_k+M]\right) = 1$ for all $k \in [K]$. Then for any $t \geq 1$, we have
    \begin{equation} \label{eq::integrability_condition_for_S_k_under_bounded_condition}
        \begin{aligned}
        	\mathbb{E}\left[\left|L(t+1) - L(t) \right|\mid \F_t\right] 
	& =  \mathbb{E}\left[ \frac{\mathbbm{1}\left(\kappa = k\right)}{N_k(\Tau)}\mathbbm{1}\left(A_{t+1} = k\right) \left|Y_{t+1} - \mu_k\right| \mid \mathcal{F}_{t} \right]  \\
	& \leq M \mathbb{E}\left[ \frac{\mathbbm{1}\left(A_{t+1} = k\right)}{N_k(\Tau)} \mid \mathcal{F}_{t} \right]. 
        \end{aligned}
    \end{equation}
    Therefore, we obtain that
    \begin{align*}
        \mathbb{E}[U] &=   \sum_{s=1}^{\infty} \mathbb{E}\left[\mathbb{E} \left[ \left|L(s) - L(s-1) \right| \mathbbm{1}\left( \Tau \geq s  \right) \mid \mathcal{F}_{s-1} \right]\right] \\
        & =    \sum_{s=1}^{\infty} \mathbb{E} \left[\mathbbm{1}\left( \Tau \geq s  \right) \mathbb{E} \left[ \left|L(s) - L(s-1) \right|  \mid \mathcal{F}_{s-1} \right]\right] ~~\text{(since $\mathbbm{1}\left( \Tau \geq s  \right) \in \mathcal{F}_{s-1}$)}\\
         & \leq   M  \sum_{s=1}^{\infty} \mathbb{E} \left[ \frac{\mathbbm{1}\left(A_{s} = k\right)}{N_k(\Tau)} \mathbbm{1}\left( \Tau \geq s  \right)\right] ~~\text{(by the inequality~\eqref{eq::integrability_condition_for_S_k_under_bounded_condition})}\\
         & =  M  < \infty~~\text{(by the definition of $N_k(\Tau)$)},
    \end{align*}
    which implies that if each arm has a bounded distribution,  we can determine the sign of the bias of the sample mean  at the stopping time $\Tau$ without assuming $\mathbb{E}N_k(\Tau) < \infty$ for all $k$ with $\mathbb{P}(\kappa = k) > 0$. 
 
\end{proof}

\paragraph{About Remark~\ref{remark::moment_condition}.}
In our recent work \citep{shin2019bias}, we showed that if arm $k$ has a finite $p$-th moment for a fixed $p > 2$, the following bound on the normalized $\ell_2$ risk of the sample mean holds:
\begin{equation}\label{eq:risk.bound}
    \mathbb{E}\left[\frac{N_{k}(\Tau)}{\log N_{k}(\Tau)} \left(\hat{\mu}_k (\Tau) - \mu_k\right)^2\right] < \infty,
\end{equation}
 provided that $N_k(\Tau) \geq 3$. In this case,
we can show that $\mathbb{E}[U] < \infty$ without assuming $\mathbb{E}N_k(\Tau) <\infty$, where $U$ is defined in \eqref{eq::U}. For each $k$, set $c_k := \int|x-\mu_k|\mathrm{d}P_k(x)$. Let $\hat{c}_k(\Tau)$ be the sample mean estimator of $c_k$ at the stopping time $\Tau$. Then, we have
\begin{align*}
      \mathbb{E}[U] &=   \sum_{s=1}^{\infty} \mathbb{E}\left[\mathbb{E} \left[ \left|L(s) - L(s-1) \right| \mathbbm{1}\left( \Tau \geq s  \right) \mid \mathcal{F}_{s-1} \right]\right] \\
         & =  \sum_{s=1}^{\infty}\mathbb{E} \left[ \frac{\mathbbm{1}\left(\kappa = k\right)}{N_k(\Tau)}\mathbbm{1}\left(A_{s} = k\right)\left|Y_s - \mu_k\right| \mathbbm{1}\left( \Tau \geq s  \right)\right]\\
         & \leq \mathbb{E} \left[\sum_{s=1}^{\infty} \frac{\mathbbm{1}\left(A_{s} = k\right)}{N_k(\Tau)}\left|Y_s - \mu_k\right| \mathbbm{1}\left( \Tau \geq s  \right)\right]\\
         &:= \mathbb{E}\left[\hat{c}_{k}(\Tau)\right] \\
        & \leq \mathbb{E}\left|\hat{c}_{k}(\Tau) - c_k\right| +c_k \\
        & \leq \mathbb{E}\left[\sqrt{\frac{N_{k}(\Tau)}{\log N_{k}(\Tau)}} \left|\hat{c}_k (\Tau) - c_k\right|\right]+ c_k \\
        & \leq \sqrt{\mathbb{E}\left[\frac{N_{k}(\Tau)}{\log N_{k}(\Tau)} \left(\hat{c}_k (\Tau) - c_k\right)^2\right]}  + c_k < \infty,
\end{align*}
where in the last bound we have used \eqref{eq:risk.bound}.
 Thus, if each arm has a finite $p$-th moment for a fixed $p >2$, we can determine the sign of the bias of the sample mean at the stopping time $\Tau$ without assuming $\mathbb{E}N_k(\Tau) < \infty$ for all $k$ with $\mathbb{P}(\kappa = k) > 0$.

\subsection{Proof of Corollary~\ref{cor::lil'UCB} (The lil'UCB algorithm results in positive bias)} \label{Appen::subSec::lilUCB_proof}

Before presenting a formal proof of Corollary~\ref{cor::lil'UCB}, we first provide an intuitive explanation why any reasonable and efficient algorithm for the best-arm identification problem would result in positive bias. For any $k \in [K]$ and $i \in \mathbb{N}$, let $\D_\infty^*$ and $\D_\infty^{*'}$ be two MAB tabular representation that agree with each other except $X_{i,k}^* < X_{i,k}^{*'}$. Since we have a larger value from arm $k$ in the second scenario $\D_{\infty}^{*'}$, if $\kappa = k$ under the first scenario $\D_{\infty}^*$, any reasonable algorithm would also pick the arm $k$ under the more favorable scenario $\D_{\infty}^{*'}$. In this case, we know that $\kappa = k$ implies $\kappa' = k$. Also note that any efficient algorithm should be able to exploit the more favorable scenario $D_{\infty}^{*'}$ to declare arm $k$ as the best arm by using less samples from arm $k$. Therefore, we would have $N_k(\Tau) \geq N_k'(\Tau')$. In sum, we can expect that, from any reasonable and efficient algorithm, we would have $    \frac{\mathbbm{1}(\kappa = k)}{N_k(\Tau)} \leq  \frac{\mathbbm{1}(\kappa' = k)}{N_k'(\Tau')}$ which shows that the algorithm would be monotonically increasing and thus the sample mean of the chosen arm is positively biased. Below, we formally verify that this intuition works for the lil'UCB algorithm.

\begin{proof}[Proof of Corollary~\ref{cor::lil'UCB}]
For any given $i,k$, let $X_{i,k}^{*'}$ be an independent copy of $X_{i,k}^*$ and define $X_{\infty}^{*'}$ as a $\mathbb{N} \times K$ table which equals $X_{\infty}^{*}$ on all entries except the $(i,k)$-th entry, which contains $X_{i,k}^{*'}$. Let $\D_{\infty}^{*'} =  X_{\infty}^{*'} \cup \{W_{-1}, W_0, \dots\}$  denote the corresponding dataset, which only differs from $\D^*_\infty$ in one element. Let $(N_k(T), N_k'(T))$ denote the numbers of draws from arm $k$ up to time $T$. Let $(\Tau, \Tau')$ be the stopping times and $(\kappa, \kappa')$ be choosing functions as determined by the lil'UCB algorithm under $\D_\infty^*$ and $\D_\infty^{*'}$ respectively. 

Suppose $X_{i,k}^{*} \leq X_{i,k}^{*'}$. Proving that the lil'UCB algorithm is monotonically increasing (and hence results in positive bias) corresponds to showing that the following inequality holds:
\begin{equation} \label{eq::lil_ineq}
    \frac{\mathbbm{1}(\kappa = k)}{N_k(\Tau)} \leq  \frac{\mathbbm{1}(\kappa' = k)}{N_k'(\Tau')}.
\end{equation}

If $\kappa \neq k$, the inequality~\eqref{eq::lil_ineq} holds trivially. Therefore, for the rest of the proof, we assume $\kappa = k$ which also implies $\Tau  < \infty$. (If not, the lil'UCB algorithm is not stopped, and thus $\kappa \neq k$.)

First, we can check that the lil'UCB sampling is a special case of UCB-type sampling algorithms. Therefore, it is an optimistic sampling method which implies that for any \emph{fixed} $t > 0$, and \emph{fixed} arm $k$, we have $N_k(t)~\leq~N_k'(t)$. Since $\sum_{j \neq k} N_j(t) = t - N_k(t)$ for all $t$, we can rewrite the lil'UCB stopping rule as stopping the sampling whenever there exists a $k$ such that $N_k$, which is a non-decreasing function of $t$, crosses the strictly increasing linear boundary $\left\{(n, t) : n = \frac{1 + \lambda t}{1 + \lambda}\right\}$ for a fixed $\lambda >0$. Since $N_k(t) \leq N_k'(t)$ for all $t$, we know that $\Tau' \leq \Tau$. 

Since the linear boundary is increasing, we can check $N_k'(\Tau') \leq N_k(\Tau)$ if $\kappa' = k$. Therefore, to complete the proof, it is enough to show that $\kappa = k$ implies $\kappa' = k$. For the sake of deriving a contradiction,  assume $\kappa = k$ but $\kappa'  \neq k$. Then, there exists $j \neq k$ such that $\kappa' = j$.  By the definition of $\kappa'$, it is equivalent to $N_j'(\Tau') = \max_{l \in [K]} N_l'(\Tau')$. Hence, we have that
\begin{equation} \label{eq::lilUCB_lemma1}
N_j'(\Tau')  > N_k'(\Tau').    
\end{equation}
Similarly, we can show that
\begin{equation} \label{eq::lilUCB_lemma2}
N_j(\Tau)  < N_k(\Tau).    
\end{equation}

Since $\Tau'$ is the first time $t$ such that, for some $l$, $N_l'(t)$ has crossed the boundary, we know that $j$ is also the index of the arm which has crossed the boundary first time. Also, since the lil'UCB sampling satisfies the IIO condition, Lemma~\ref{lemma::cor_of_IIO} along with the fact that $N_k(t) \leq N_k'(t)$ for all $t$ implies that $N_j(t) \geq N_j'(t)$ for all $j \neq k$. From the two observations above, we have the following inequalities: 
\[
\frac{1 + \lambda \Tau'}{1+ \lambda} \leq N_j'(\Tau')  \leq N_j(\Tau'), 
\]
which implies that $t \mapsto N_j(t)$ is crossing the boundary at time $\Tau'$. By the definition of $\Tau$ and, by assumption, $\kappa = k$, we obtain that $\Tau \leq \Tau'$.

Similarly, from the fact that $N_k(t) \leq N_k'(t)$ for all $t$ along with the definition of $\Tau$, we have that
\[
\frac{1 + \lambda \Tau}{1+ \lambda} \leq N_k(\Tau)  \leq N_k'(\Tau),
\]
which implies that $t \mapsto N_k'(t)$ is crossing the boundary at time $\Tau$, and thus $\Tau' \leq \Tau$ since $\kappa' \neq k$ by assumption.

From the two observations above, we have $\Tau' = \Tau$. Finally, note that 
\[
 N_k'(\Tau') < N_j'(\Tau') \leq N_j(\Tau') = N_j(\Tau) < N_k(\Tau) \leq N_k'(\Tau) = N_k'(\Tau')
\]
where the first inequality comes from the inequality~\eqref{eq::lilUCB_lemma1}. The second inequality come from $N_j' \leq N_j$. The first equality comes from $\Tau' = \Tau$ and the third inequality comes from the inequality~\eqref{eq::lilUCB_lemma2}. The last inequality comes from $N_k \leq N_k'$ and the final equality comes from $\Tau = \Tau'$.

This is a contradiction, and, therefore, $\kappa = k$  implies that $\kappa' = k$. This proves that the lil'UCB algorithm is monotonically increasing and the chosen stopped sample mean from the lil'UCB algorithm is positively biased. 
\end{proof}

\subsection{Proof of Proposition~\ref{prop::bias_expression} (bias expression) via Lemma~\ref{lem:Wald-gen} (Wald's identity for MAB)}
\label{subSec::prop_bias_expression}
% \begin{proof}[Proof of Proposition~\ref{prop::bias_expression}]
   % By the Wald's equation, we have $\mu_k \mathbb{E}\left[N_k(\Tau)\right] = \mathbb{E}\left[S_k(\Tau)\right]$. 
By direct substitution, we first note that 
    \begin{align*}
        \mathbb{E}\left|S_k(\Tau) - \mu_k N_k(\Tau)\right|
        & = \mathbb{E} \left[\sum_{t=1}^\infty \mathbbm{1}\left(A_t = k\right) \left|Y_t - \mu_k\right| \mathbbm{1}\left(\Tau \geq t\right)\right]\\
        & = \sum_{t=1}^\infty \mathbb{E}  \left[ \mathbbm{1}\left(A_t = k\right) \left|Y_t - \mu_k\right| \mathbbm{1}\left(\Tau \geq t\right)\right]\\
        & = \sum_{t=1}^\infty \mathbb{E}  \left[ \mathbbm{1}\left(A_t = k\right)\mathbbm{1}\left(\Tau \geq t\right) \mathbb{E}\left[|Y_t - \mu_k| \mid \mathcal{F}_{t-1} \right]  \right]\\
        & = \sum_{t=1}^\infty \mathbb{E}  \left[ \mathbbm{1}\left(A_t = k\right)\mathbbm{1}\left(\Tau \geq t\right)  \int |x-\mu_k| \mathrm{d}P_k(x)  \right]\\
        & = \int |x-\mu_k| \mathrm{d}P_k(x) \mathbb{E}  \left[  \sum_{t=1}^\infty  \mathbbm{1}\left(A_t = k\right)\mathbbm{1}\left(\Tau \geq t\right)   \right]\\
        & = \int |x-\mu_k| \mathrm{d}P_k(x) \mathbb{E}  \left[N_k(\Tau)  \right] <\infty,
    \end{align*}
    where the second equality comes from the Tonelli's theorem and the third equality stems from the facts that $\mathbbm{1}(A_t =k)$ and $\mathbbm{1}(\Tau \geq t)$ are $\mathcal{F}_{t-1}$ measurable. The fourth equality comes from the fact that, on event $\mathbbm{1}(A_t =k)$, $Y_t \sim P_k$ and it is independent of the previous history. Finally, the finiteness of the last term comes from the assumption of the existence of the first moment of $k$-th arm and $\mathbb{E}[N_k(\Tau)] <\infty$. Therefore, by the dominated convergence theorem, we have 
    \begin{align*}
      \mathbb{E}\left[S_k(\Tau) - \mu_k N_k(\Tau)\right]
        & = \mathbb{E} \left[\sum_{t=1}^\infty \mathbbm{1}\left(A_t = k\right) \left[Y_t - \mu_k\right] \mathbbm{1}\left(\Tau \geq t\right)\right]\\
        & = \sum_{t=1}^\infty \mathbb{E}  \left[ \mathbbm{1}\left(A_t = k\right) \left[Y_t - \mu_k\right] \mathbbm{1}\left(\Tau \geq t\right)\right]\\
        & = \sum_{t=1}^\infty \mathbb{E}  \left[ \mathbbm{1}\left(A_t = k\right)\mathbbm{1}\left(\Tau \geq t\right) \mathbb{E}\left[Y_t - \mu_k \mid \mathcal{F}_{t-1} \right]  \right]\\
        & = 0,
    \end{align*}
    which implies $\mu_k \mathbb{E}\left[N_k(\Tau)\right] = \mathbb{E}\left[S_k(\Tau)\right]$, which proves the generalization of Wald's first identity.

    Since $\mathbb{E}\left[N_k(\Tau)\right] >0$, one can then express $\mu_k$ as
	\[
	\mu_k = \frac{\mathbb{E}\left[S_k(\Tau)\right] }{\mathbb{E}[N_k(\Tau)]}.
	\]
	By direct substitution, the bias of the sample mean can thus be expressed as
	\begin{align*}
	\mathbb{E}\left[\hat{\mu}_k(\Tau) - \mu_k\right] &= \mathbb{E}\left[\hat{\mu}_k(\Tau) \left(1 -  \frac{N_k(\Tau)}{\mathbb{E}[N_k(\Tau)]} \right) \right] \\
	&= \mathrm{Cov}\left(\hat{\mu}_k(\Tau), \left( 1 -  \frac{N_k(\Tau)}{\mathbb{E}[N_k(\Tau)]} \right)\right) \\
	&= -\frac{\mathrm{Cov}\left(\hat{\mu}_k(\Tau), N_k(\Tau)\right)}{\mathbb{E}[N_k(\Tau)]}.
	\end{align*}
This completes the proof of the proposition.
% \end{proof}

\section{Additional simulation results}

\subsection{More on negative bias due to optimistic sampling}\label{appen::subSec::simul_1}

We conduct a simulation study in which we have three unit-variance Gaussian arms with $\mu_1 = 1, \mu_2 = 2$ and $\mu_3 =3$. After sampling once from each arm, greedy, UCB and Thompson sampling are used to continue sampling until $T = 200$. We repeat the whole process from scratch  $10^4$ times for each algorithm to get an accurate estimate for the bias.

For UCB, we use $u_{t-1}(s, n) = \sqrt{\frac{2\log(1/\delta)}{n}}$ with $\delta = 0.1$. For Thompson sampling, we use independent standard Normal priors for simplicity. We repeat the whole process from scratch  $2000$ times for each algorithm to get an accurate estimate for the bias. 

Figure~\ref{fig::simul_1_greedy} shows the distribution of observed differences between sample means and the true mean for each arm under the greedy algorithm. Vertical lines correspond to biases. The example demonstrates that the sample mean is negatively biased under optimistic sampling rules. Similar results from UCB / Thompson sampling algorithms can be found in Section~\ref{subSec::simul_1}.

\begin{figure}[!ht]
\begin{center}
\includegraphics[scale=.7]{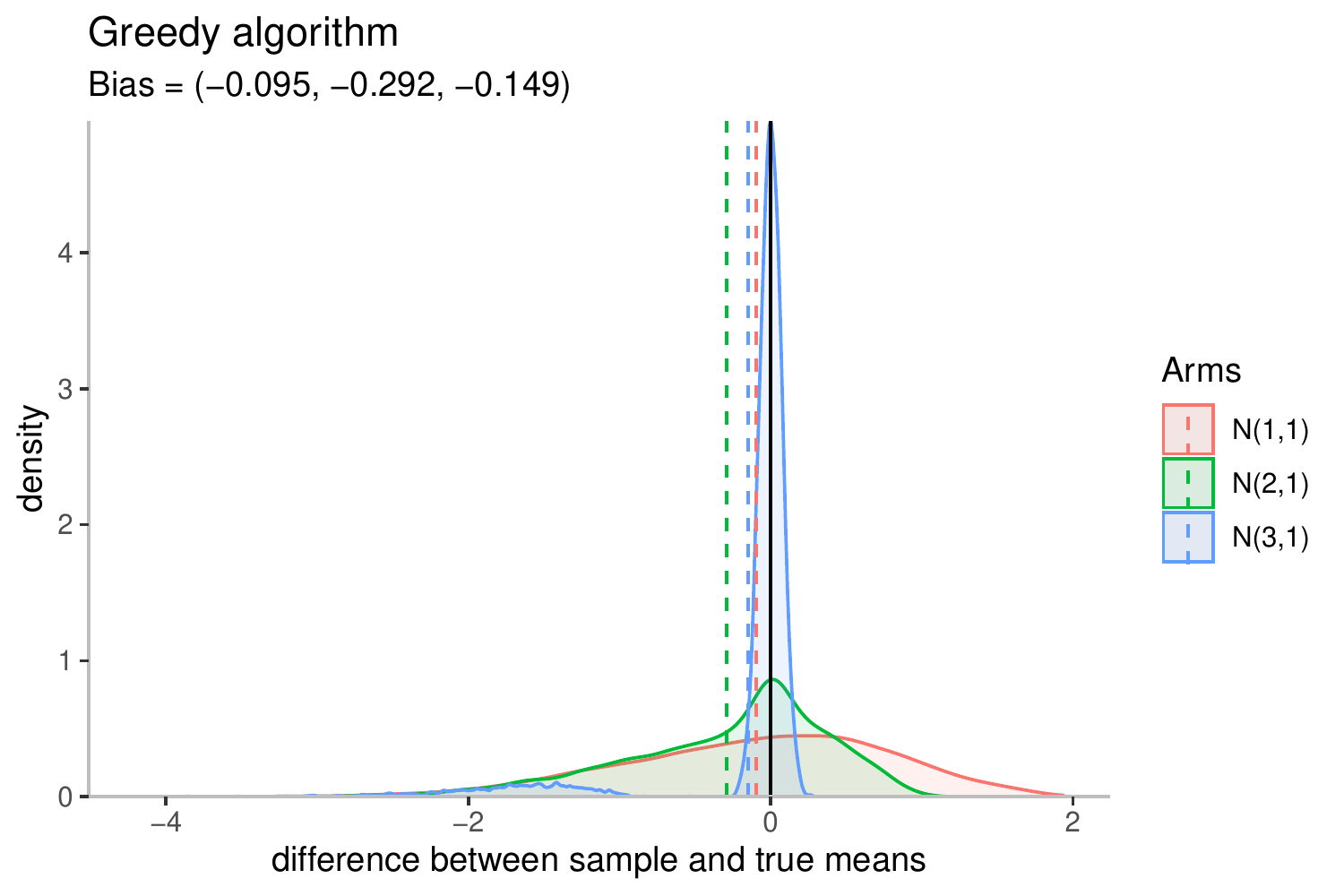}
\end{center}
	\caption{\em  Data is collected by the greedy algorithm from three unit-variance Gaussian arms with $\mu_1 =1, \mu_2 = 2$ and $\mu_3 = 3$. For all three arms, sample means are negatively biased.}
	\label{fig::simul_1_greedy}
\end{figure}

\subsection{Positive bias from optimistic choosing and stopping in identifying the largest mean}
\label{appen::subSec::simulation_3}

Suppose we have $K$ arms with mean $\mu_1, \dots, \mu_K$. As we were in Section~\ref{subSec::simulation_4}, we are interested not in each individual arm but in the arm with the largest mean. That is, our target of inference is $\mu_* := \max_{k \in [K]} \mu_k$.

Instead of using the lil'UCB algorithm, we can draw a sample from each arm in a cyclic order for each time $t$ and use 
a naive sequential procedure based on the following stopping time.
\begin{equation}
    \Tau_{M}^\delta := \inf\left\{t \in \{K, 2K,\dots, MK\}: \hat{\mu}_{(1)}(t) > \hat{\mu}_{(2)}(t) + \delta  \right\},
\end{equation}
where $M, \delta > 0$ are prespecified constants and $\hat{\mu}_{(k)}(t)$ is the $k$-th largest sample mean at time $t$. Once we stop sampling at time $\Tau_{M}^\delta $, we can estimate the largest mean by the largest stopped sample mean $\hat{\mu}_{(1)}\left(\Tau_{M}^\delta \right)$. 

The performance of this sequential procedure can vary based on underlying distribution of the arm and the choice of $\delta$ and $M$. However, we can check this optimistic choosing and stopping rules are jointly monotonic increasing and thus the largest stopped sample mean $\hat{\mu}_{(1)}\left(\Tau_{M}^\delta \right)$ is always positively based for any choice of $\delta$ and $M$.

To verify it with a simulation, we set $3$ unit-variance Gaussian arms with means $(\mu_1, \mu_2, \mu_3) = (g, 0, -g)$ for each gap parameter $g = 1, 3, 5$. We conduct $10^4$ trials of this sequential procedure with $M = 1000$ and $\delta = 0.7 \times g$. Figure~\ref{fig::simul_3} shows the distribution of observed differences between the chosen sample means and the corresponding true mean for each $\delta$. Vertical lines correspond to biases. The simulation study demonstrate that, in all configurations, the largest stopped sample mean $\hat{\mu}_{(1)}\left(\Tau_{M}^\delta \right)$ is always positively biased. Note, in contrast to the lil'UCB case in Section~\ref{subSec::simulation_4}, we have a larger bias for a smaller gap since the number of sample sizes are similar for each gaps due to the adaptive (and oracle) choice of the parameter $\delta$ but a smaller gap makes more difficult to identify largest mean correctly.

\begin{figure}[h!]
	\begin{center}
	\includegraphics[scale =  0.6]{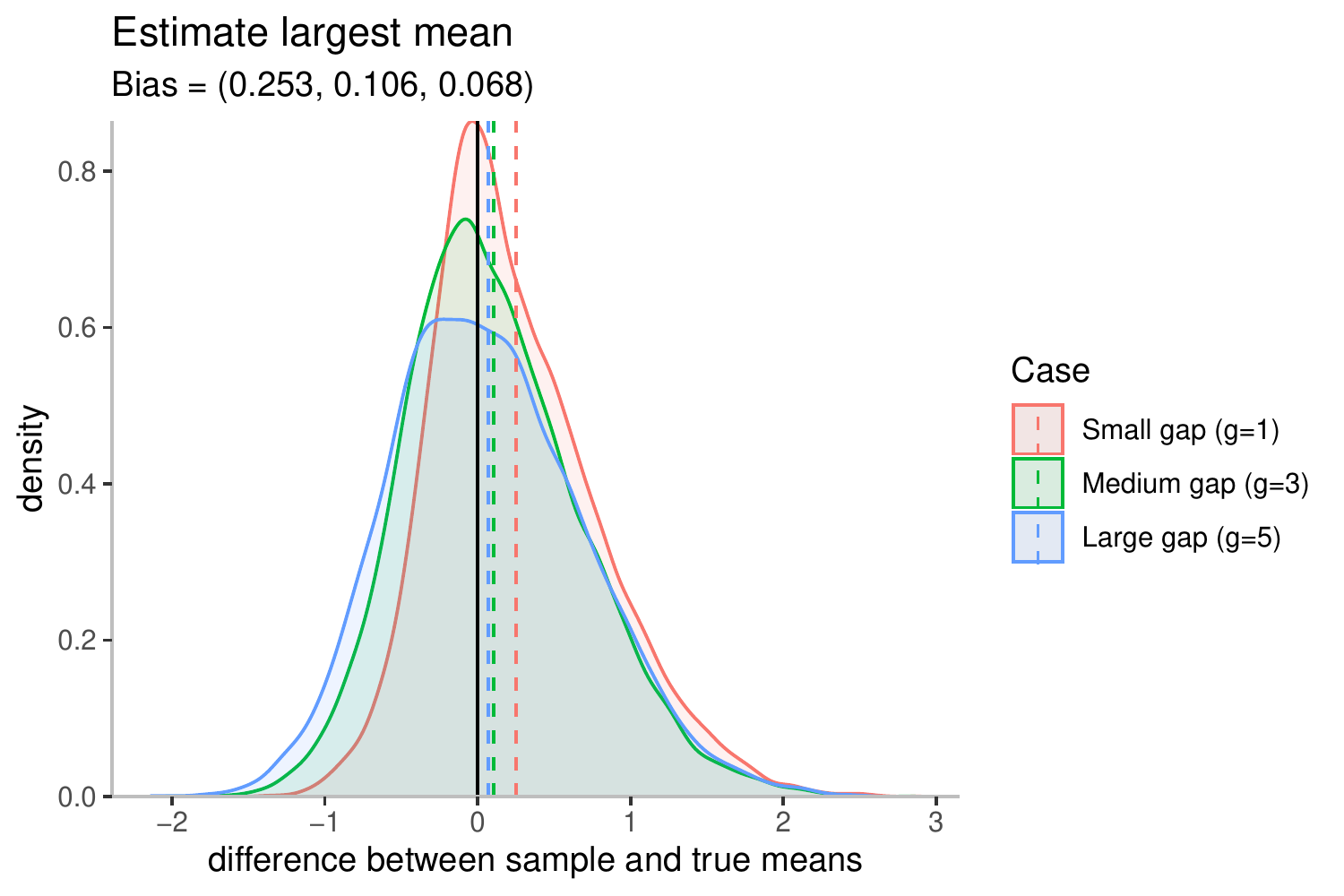}
	\end{center}
	\caption{\em  Data is collected by the sequential procedure described in Appendix~\ref{appen::subSec::simulation_3} under unit-variance Gaussian arms with $\mu_1 = g, \mu_2 = 0$ and $\mu_3 = -g$ for each gap parameter $g = \{1, 3, 5\}$. For each gap $g$, we set the parameter $\delta = 0.7 \times g$ and $M = 1000$. For all cases, chosen sample means are positively biased. }
	\label{fig::simul_3}
\end{figure}

%% file: neurips_2019.bbl
\begin{thebibliography}{23}
\providecommand{\natexlab}[1]{#1}
\providecommand{\url}[1]{\texttt{#1}}
\expandafter\ifx\csname urlstyle\endcsname\relax
  \providecommand{\doi}[1]{doi: #1}\else
  \providecommand{\doi}{doi: \begingroup \urlstyle{rm}\Url}\fi

\bibitem[Agrawal and Goyal(2012)]{agrawal2012analysis}
Shipra Agrawal and Navin Goyal.
\newblock Analysis of thompson sampling for the multi-armed bandit problem.
\newblock In \emph{Conference on Learning Theory}, pages 39--1, 2012.

\bibitem[Audibert and Bubeck(2009)]{audibert2009minimax}
Jean-Yves Audibert and S{\'e}bastien Bubeck.
\newblock Minimax policies for adversarial and stochastic bandits.
\newblock In \emph{COLT}, pages 217--226, 2009.

\bibitem[Auer et~al.(2002)Auer, Cesa-Bianchi, and Fischer]{auer2002finite}
Peter Auer, Nicolo Cesa-Bianchi, and Paul Fischer.
\newblock Finite-time analysis of the multiarmed bandit problem.
\newblock \emph{Machine learning}, 47\penalty0 (2-3):\penalty0 235--256, 2002.

\bibitem[Bowden and Trippa(2017)]{bowden2017unbiased}
Jack Bowden and Lorenzo Trippa.
\newblock Unbiased estimation for response adaptive clinical trials.
\newblock \emph{Statistical methods in medical research}, 26\penalty0
  (5):\penalty0 2376--2388, 2017.

\bibitem[Garivier and Capp{\'e}(2011)]{garivier2011kl}
Aur{\'e}lien Garivier and Olivier Capp{\'e}.
\newblock The {KL-UCB} algorithm for bounded stochastic bandits and beyond.
\newblock In \emph{Proceedings of the 24th Annual Conference On Learning
  Theory}, pages 359--376, 2011.

\bibitem[Garivier and Kaufmann(2016)]{garivier2016optimal}
Aur{\'e}lien Garivier and Emilie Kaufmann.
\newblock Optimal best arm identification with fixed confidence.
\newblock In \emph{Conference on Learning Theory}, pages 998--1027, 2016.

\bibitem[Gut(2009)]{gut2009stopped}
Allan Gut.
\newblock \emph{Stopped random walks}.
\newblock Springer, 2009.

\bibitem[Howard et~al.(2018)Howard, Ramdas, McAuliffe, and
  Sekhon]{howard2018uniform}
Steven~R Howard, Aaditya Ramdas, Jon McAuliffe, and Jasjeet Sekhon.
\newblock Uniform, nonparametric, non-asymptotic confidence sequences.
\newblock \emph{arXiv preprint arXiv:1810.08240}, 2018.

\bibitem[Jamieson et~al.(2014)Jamieson, Malloy, Nowak, and
  Bubeck]{jamieson_lil_2014}
Kevin Jamieson, Matthew Malloy, Robert Nowak, and Sébastien Bubeck.
\newblock lil' {UCB}: {An} {Optimal} {Exploration} {Algorithm} for
  {Multi}-{Armed} {Bandits}.
\newblock In \emph{Proceedings of {The} 27th {Conference} on {Learning}
  {Theory}}, volume~35 of \emph{Proceedings of {Machine} {Learning}
  {Research}}, pages 423--439, 2014.

\bibitem[Kalyanakrishnan et~al.(2012)Kalyanakrishnan, Tewari, Auer, and
  Stone]{kalyanakrishnan2012pac}
Shivaram Kalyanakrishnan, Ambuj Tewari, Peter Auer, and Peter Stone.
\newblock Pac subset selection in stochastic multi-armed bandits.
\newblock In \emph{ICML}, volume~12, pages 655--662, 2012.

\bibitem[Kaufmann et~al.(2012)Kaufmann, Korda, and Munos]{kaufmann2012thompson}
Emilie Kaufmann, Nathaniel Korda, and R{\'e}mi Munos.
\newblock Thompson sampling: An asymptotically optimal finite-time analysis.
\newblock In \emph{International Conference on Algorithmic Learning Theory},
  pages 199--213. Springer, 2012.

\bibitem[Lattimore and Szepesv{\'a}ri(2019)]{lattimore2018bandit}
Tor Lattimore and Csaba Szepesv{\'a}ri.
\newblock Bandit algorithms.
\newblock \emph{Cambridge University Press}, 2019.

\bibitem[Nie et~al.(2018)Nie, Tian, Taylor, and Zou]{nie2018adaptively}
Xinkun Nie, Xiaoying Tian, Jonathan Taylor, and James Zou.
\newblock Why adaptively collected data have negative bias and how to correct
  for it.
\newblock In \emph{International Conference on Artificial Intelligence and
  Statistics}, pages 1261--1269, 2018.

\bibitem[Robbins(1952)]{robbins1952some}
Herbert Robbins.
\newblock Some aspects of the sequential design of experiments.
\newblock \emph{Bulletin of the American Mathematical Society}, 58\penalty0
  (5):\penalty0 527--535, 1952.

\bibitem[Robbins(1970)]{robbins1970statistical}
Herbert Robbins.
\newblock Statistical methods related to the law of the iterated logarithm.
\newblock \emph{The Annals of Mathematical Statistics}, 41\penalty0
  (5):\penalty0 1397--1409, 1970.

\bibitem[Russo(2016)]{russo2016simple}
Daniel Russo.
\newblock Simple bayesian algorithms for best arm identification.
\newblock In \emph{Conference on Learning Theory}, pages 1417--1418, 2016.

\bibitem[Shin et~al.(2019)Shin, Ramdas, and Rinaldo]{shin2019bias}
Jaehyeok Shin, Aaditya Ramdas, and Alessandro Rinaldo.
\newblock On the bias, risk and consistency of sample means in multi-armed
  bandits.
\newblock \emph{arXiv preprint arXiv:1902.00746}, 2019.

\bibitem[Siegmund(1978)]{siegmund1978estimation}
David Siegmund.
\newblock Estimation following sequential tests.
\newblock \emph{Biometrika}, 65\penalty0 (2):\penalty0 341--349, 1978.

\bibitem[Starr and Woodroofe(1968)]{starr1968remarks}
Norman Starr and Michael~B Woodroofe.
\newblock Remarks on a stopping time.
\newblock \emph{Proceedings of the National Academy of Sciences of the United
  States of America}, 61\penalty0 (4):\penalty0 1215, 1968.

\bibitem[Sutton and Barto(1998)]{sutton1998introduction}
Richard~S Sutton and Andrew~G Barto.
\newblock \emph{Introduction to reinforcement learning}.
\newblock MIT press Cambridge, 1998.

\bibitem[Thompson(1933)]{thompson1933likelihood}
William~R Thompson.
\newblock On the likelihood that one unknown probability exceeds another in
  view of the evidence of two samples.
\newblock \emph{Biometrika}, 25\penalty0 (3/4):\penalty0 285--294, 1933.

\bibitem[Villar et~al.(2015)Villar, Bowden, and Wason]{villar2015multi}
Sof{\'\i}a~S Villar, Jack Bowden, and James Wason.
\newblock Multi-armed bandit models for the optimal design of clinical trials:
  benefits and challenges.
\newblock \emph{Statistical science: a review journal of the Institute of
  Mathematical Statistics}, 30\penalty0 (2):\penalty0 199, 2015.

\bibitem[Xu et~al.(2013)Xu, Qin, and Liu]{xu2013estimation}
Min Xu, Tao Qin, and Tie-Yan Liu.
\newblock Estimation bias in multi-armed bandit algorithms for search
  advertising.
\newblock In \emph{Advances in Neural Information Processing Systems}, pages
  2400--2408, 2013.

\end{thebibliography}
